\documentclass[11pt]{article}

\usepackage[a4paper,
            includefoot,
            left=2cm, right=1.5cm, top=2cm, bottom=2cm, headsep=1cm, footskip=1cm]{geometry}
\usepackage[x11names]{xcolor}
\usepackage[unicode=true,pdfusetitle,
 bookmarks=true,bookmarksnumbered=false,bookmarksopen=true,bookmarksopenlevel=1,
 breaklinks=false,pdfborder={0 0 1},backref=false]
 {hyperref}
 \hypersetup{
    colorlinks = false,
    linkbordercolor  = green,
    citebordercolor = red,
 }

\usepackage{lipsum}
\usepackage{bm}
\usepackage{amsmath}
\usepackage{amssymb}
\usepackage{amsthm}
\usepackage{graphicx}
\usepackage{epstopdf}

\ifpdf
  \DeclareGraphicsExtensions{.eps,.pdf,.png,.jpg}
\else
  \DeclareGraphicsExtensions{.eps}
\fi

\usepackage{amsmath}
\numberwithin{equation}{section}

\usepackage{amsopn}
\DeclareMathOperator{\diag}{diag}
\DeclareMathOperator{\mes}{mes}

\usepackage{euscript}
\usepackage{relsize}
\usepackage{mathdots}
\usepackage{graphicx}
\usepackage{indentfirst}
\usepackage{empheq}
\usepackage{multirow}
\usepackage{thmtools}
\usepackage{thm-restate}
\usepackage{tikz}
\usepackage{algorithm}
\usepackage{algpseudocode}

\DeclareMathOperator*{\argmax}{arg\,max}
\DeclareMathOperator*{\argmin}{arg\,min}

\newcommand{\row}[1]{#1 ,\,\boldsymbol{:}\,}
\newcommand{\vecrow}[1]{#1}
\newcommand{\col}[1]{\,\boldsymbol{:}, #1}

\def\rank{\mathop{\mathrm{rank}}}

\def\colspace{\mathop{\mathrm{colspace}}}
\def\rowspace{\mathop{\mathrm{rowspace}}}

\def\Sigminus{\bfW}
\newcommand\winverse[2]{\left(#1 \right)^{\dag}_{#2}}
\newcommand\inverse[1]{#1^{\dag}}
\def\i0{\tau}
\def\Si0{\dot \bfs}
\def\Ai0{\dot \bfa}
\def\Proj{\mathbf{\Pi}}

\def\fullop{H_{\tau}}

\newcommand{\Input}{\hspace*{\algorithmicindent} \textbf{Input}:\ }
\newcommand\upangle[2]{\mathord{<\mspace{-9mu}\mathrel{)}\mspace{2mu}}(#1, #2)}

\newtheorem{proposition}{Proposition}[section]
\newtheorem{theorem}{Theorem}[section]
\newtheorem{remark}{Remark}[section]
\newtheorem{lemma}{Lemma}[section]

\usepackage{euscript}

\newcommand{\spC}{\mathbb{C}}

\newcommand{\spR}{\mathbb{R}}

\newcommand{\spT}{\mathbb{T}}

\newcommand{\tsA}{\mathsf{A}}
\newcommand{\tsB}{\mathsf{B}}
\newcommand{\tsC}{\mathsf{C}}
\newcommand{\tsD}{\mathsf{D}}

\newcommand{\tsN}{\mathsf{N}}

\newcommand{\tsS}{\mathsf{S}}

\newcommand{\tsX}{\mathsf{X}}
\newcommand{\tsY}{\mathsf{Y}}
\newcommand{\tsZ}{\mathsf{Z}}

\newcommand{\bfA}{\mathbf{A}}
\newcommand{\bfB}{\mathbf{B}}
\newcommand{\bfC}{\mathbf{C}}
\newcommand{\bfG}{\mathbf{G}}
\newcommand{\bfL}{\mathbf{L}}
\newcommand{\bfX}{\mathbf{X}}
\newcommand{\bfF}{\mathbf{F}}
\newcommand{\bfY}{\mathbf{Y}}
\newcommand{\bfZ}{\mathbf{Z}}
\newcommand{\bfM}{\mathbf{M}}
\newcommand{\bfP}{\mathbf{P}}
\newcommand{\bfQ}{\mathbf{Q}}
\newcommand{\bfT}{\mathbf{T}}
\newcommand{\bfV}{\mathbf{V}}
\newcommand{\bfW}{\mathbf{W}}
\newcommand{\bfI}{\mathbf{I}}
\newcommand{\bfR}{\mathbf{R}}
\newcommand{\bfS}{\mathbf{S}}
\newcommand{\bfU}{\mathbf{U}}
\newcommand{\bfO}{\mathbf{O}}
\newcommand{\bfzero}{\mathbf{0}}

\newcommand{\bfJ}{\mathbf{J}}

\newcommand{\rmT}{\mathrm{T}}

\newcommand{\calB}{\mathcal{B}}
\newcommand{\calD}{\mathcal{D}}
\newcommand{\calC}{\mathcal{C}}
\newcommand{\calT}{T}
\newcommand{\calF}{\mathcal{F}}
\newcommand{\calJ}{\mathcal{J}}
\newcommand{\calH}{\mathcal{H}}
\newcommand{\calL}{\mathcal{L}}
\newcommand{\calM}{\mathcal{M}}

\newcommand{\calG}{\mathcal{G}}
\newcommand{\calK}{\mathcal{K}}
\newcommand{\calI}{\mathcal{I}}
\newcommand{\calQ}{\mathcal{Q}}
\newcommand{\calS}{\mathcal{S}}

\newcommand{\calW}{\mathcal{W}}
\newcommand{\calZ}{\mathcal{Z}}

\newcommand{\bfa}{\mathbf{a}}
\newcommand{\bfb}{\mathbf{b}}
\newcommand{\bfc}{\mathbf{c}}
\newcommand{\bfd}{\mathbf{d}}
\newcommand{\bfe}{\mathbf{e}}

\newcommand{\bfp}{\mathbf{p}}
\newcommand{\bfq}{\mathbf{q}}

\newcommand{\bfs}{\mathbf{s}}

\newcommand{\bfu}{\mathbf{u}}
\newcommand{\bfv}{\mathbf{v}}

\newcommand{\bfx}{\mathbf{x}}
\newcommand{\bfy}{\mathbf{y}}

\newcommand{\bt}{\begin{theorem}}
\newcommand{\et}{\end{theorem}}
\newcommand{\bl}{\begin{lemma}}
\newcommand{\el}{\end{lemma}}
\newcommand{\bp}{\begin{proposition}}
\newcommand{\ep}{\end{proposition}}
\newcommand{\bc}{\begin{corollary}}
\newcommand{\ec}{\end{corollary}}

\newcommand{\bd}{\begin{definition}\rm}
\newcommand{\ed}{\end{definition}}
\newcommand{\bex}{\begin{example}\rm}
\newcommand{\eex}{\end{example}}
\newcommand{\br}{\begin{remark}\rm}
\newcommand{\er}{\end{remark}}

\newcommand{\btbh}{\begin{table}[!ht]}
\newcommand{\etb}{\end{table}}
\newcommand{\bfgh}{\begin{figure}[!ht]}
\newcommand{\efg}{\end{figure}}

\newcommand{\bea}{\begin{eqnarray*}}
\newcommand{\eea}{\end{eqnarray*}}
\newcommand{\be}{\begin{eqnarray}}
\newcommand{\ee}{\end{eqnarray}}
\def\sspan{\mathop{\mathrm{span}}}
\def\rank{\mathop{\mathrm{rank}}}
\def\span{\mathop{\mathrm{span}}}

\newcommand{\Arg}{\mathop\mathrm{Arg}}

\makeatletter
\def\adots{\mathinner{\mkern2mu\raise\p@\hbox{.}
\mkern2mu\raise4\p@\hbox{.}\mkern1mu
\raise7\p@\vbox{\kern7\p@\hbox{.}}\mkern1mu}}
\newcommand{\l@abcd}[2]{\hbox to\textwidth{#1\dotfill #2}}
\makeatother

\def\unit{\mathfrak{i}}

\def\code#1{\texttt{#1}}

\begin{document}

\title{Image space projection for low-rank signal estimation:\\ Modified Gauss-Newton method}

\author{Nikita Zvonarev\footnote{{Faculty of Mathematics and Mechanics}, {St.Petersburg State University}, {Universitetskaya nab. 7/9, St.Petersburg, 199034}, {Russia}, nikitazvonarev@gmail.com}, Nina Golyandina\footnote{{Faculty of Mathematics and Mechanics}, {St.Petersburg State University}, {Universitetskaya nab. 7/9, St.Petersburg, 199034}, {Russia}, n.golyandina@spbu.ru, nina@gistatgroup.com}}

\maketitle

\abstract{\small The paper is devoted to the solution of a weighted nonlinear least-squares problem for low-rank signal estimation, which is related to Hankel structured low-rank approximation problems. A modified weighted Gauss-Newton method, which uses projecting on the image space of the signal, is proposed to solve this problem. The advantage of the proposed method is the possibility of its numerically stable and fast implementation. For a weight matrix, which corresponds to an autoregressive process of order $p$, the computational cost of iterations is $O(N r^2 + N p^2 + r N \log N)$, where $N$ is the time series length, $r$ is the rank of the approximating time series. For developing the method, some useful properties of the space of time series of rank $r$ are studied. The method is compared with state-of-the-art methods based on  the  variable projection approach in terms of numerical stability, accuracy and computational cost.
}

\section{Introduction}\label{sec:intro}
In this study we consider  the `signal plus noise' observation scheme:
$$
x_n=s_n+\epsilon_n, \;\;\; n=1,2, \ldots, N.
$$
Denote by
$\tsX = (x_1, \ldots, x_N)^\rmT$ , $\tsS = (s_1, \ldots, s_N)^\rmT$  and $\bm\epsilon = (\epsilon_1, \ldots, \epsilon_N)^\rmT $
the vectors of observations,  signal values and errors respectively.
We will refer to vectors of observations in $\spR^N$ as time series (or shortly series, since the observations are not necessarily
temporal; e.g., they can be spatial).

We assume that the signal $\tsS$ can  be written in the parametric form as a finite sum
    \begin{equation}
    \label{eq:model}
    s_n = \sum_{k=1}^d P_{m_k}(n)\exp(\alpha_k n) \sin(2\pi \omega_k n + \phi_k),
    \end{equation}
 where   $P_{m_k}(n)$ are polynomials in $n$ of degree $m_k$.
 In signal processing applications, the signal in the model  \eqref{eq:model} is usually a sum of sine waves \cite{Cadzow1988}
 or a sum of damped sinusoids \cite{Markovsky2008}.
 The problem of estimating the unknown signal values $s_n$ is as important as the problem of estimating the parameters in the explicit form \eqref{eq:model}.
    Both problems can be solved by the same approach, but we are concentrated on the signal estimation using a different parameterization which is wider than the explicit parameterization given in \eqref{eq:model}.

    Let $\calS\in \spR^N$ be a set, which contains a class of signals in the form \eqref{eq:model} of low complexity (to be defined later).  Consider the weighted least-squares problem (WLS) with a positive definite symmetric weight matrix $\bfW \in \spR^{N \times N}$:
	\begin{equation}
    \label{eq:wls_gen}
	\tsY^\star = \argmin_{\tsY \in \calS} \| \tsX - \tsY \|_{\bfW},
	\end{equation}
    where $\|\tsZ\|_{\bfW}^2 = \tsZ^\rmT \bfW \tsZ$.
 If noise $\bm\epsilon$ is Gaussian with covariance matrix $\bm\Sigma$ and zero mean, the WLS estimate with the weight matrix $\bfW = \bm\Sigma^{-1}$ is the maximum likelihood estimator (MLE). The same is true if the covariance matrix is scaled by a constant.

  Let us consider different approaches for solving  \eqref{eq:wls_gen}.
The chances for success in solving  problems of this kind depend on the  parameterization of the problem.
For the search of parameters in \eqref{eq:model} by the parametric
least-squares method (non-linear parametric regression), one should fix an explicit parametric form of \eqref{eq:model} in $\calS$.
Here we consider another approach to the choice of $\calS$ and its parameterization, based on the so-called signal rank,
which, in a sense, represents the signal complexity; that is, we say about the low complexity of a signal if its rank is not large.

Let us introduce some definitions.
       The rank of a signal $\tsS$  is defined as follows.
    For a given integer $L$ called the window length, we define the embedding operator $\calT_{L}:\; \spR^{N} \to \spR^{L \times (N-L+1)}$,
    which maps $\tsS$ into a Hankel $L\times (N-L+1)$ matrix, by
    	\begin{equation}
\label{eq:embedding}
	\calT_{L}(\tsS) = \begin{pmatrix}
	s_1 & s_2 & \hdots & s_{N-L+1} \\
	s_2 & s_3 & \hdots & \vdots \\
	\vdots & \vdots & \hdots & s_{N-1} \\
	s_{L} & s_{L+1} & \hdots & s_{N}
	\end{pmatrix}.
	\end{equation}
The columns of $\calT_{L}(\tsS)$ are sequential lagged vectors; this is why $\calT_{L}(\tsS)$ is often called the trajectory matrix of $\tsS$.
		We say that the signal $\tsS$ has rank $r < N/2$ if $\rank \calT_{r+1}(\tsS) = r$.
 It is known that $\rank \calT_{r+1}(\tsS) = r$ if and only if $\rank \calT_{L}(\tsS) = r$ for any $L$ such that $\min(L, N-L+1) > r$ (see \cite[Corollary 5.1]{Heinig1984} for the proof).
	
For a sufficiently large time series length $N$, the signal in the form \eqref{eq:model} has rank $r$, which is determined by the parameters $m_k$, $\alpha_k$ and
$\omega_k$
(see Section~\ref{sec:rank_calc} for explaining the correspondence between the form of \eqref{eq:model} and the rank $r$).
For example, the signal with values $s_n$ has rank $r=2$ for a sum of two exponentials $s_n=c_1 \exp(\alpha_1 n) + c_2 \exp(\alpha_2 n)$, a sine wave $s_n=c \sin(2 \pi \omega n + \phi)$, where $0<\omega<0.5$, or a linear function $s_n =a n + b$.

    Let us consider the set $\calS$ in \eqref{eq:wls_gen}, which fixes the rank $r$ but does not fix the form of the signal,
    i.e., the number of terms and degrees of polynomials in \eqref{eq:wls_gen}.
   The model of signals, where the Hankel matrix $\calT_L(\tsS)$ is rank-deficient, is one of the standard models in many areas, signal processing \cite{Cadzow1988,Tufts1993}, speech recognition \cite{Dendrinos1991}, control theory and linear systems \cite{Markovsky2008,Markovsky2019} among others.

    Denote $\calD_r$ the set of series of rank $r$.
	Since the set $\calD_r$ is not closed, we will seek for the solution of \eqref{eq:wls_gen} in its closure, i.e., $\calS=\overline{\calD_r}$. It is well-known that $\overline{\calD_r}$ consists of series of rank not larger than $r$ (this result can be found in \cite[Remark 1.46]{iarrobino1999power} for the complex case; the real-valued case is considered in Section~\ref{sec:closure}).

Thus, in what follows, we study the problem
	\begin{equation}
    \label{eq:wls}
	\tsY^\star = \argmin_{\tsY \in \overline{\calD_r}} \| \tsX - \tsY \|_{\bfW}.
	\end{equation}

Although the common case is the case of a positive definite matrix $\bfW$, the problem \eqref{eq:wls_gen}, where
 $\bfW$ is positive semi-definite, is of considerable interest. For example, the case of a diagonal matrix $\bfW$
 with several zero diagonal elements corresponds to the problem of low-rank approximation for time series with missing values if the noise is white. Let us consider the case of a general weight matrix and time series with missing values.  Let a positive definite matrix $\bfW_0$ be given for the whole time series including gaps. Then the weight matrix $\bfW$ is constructed from $\bfW_0$ by the change of columns and rows with numbers equal to entries of missing values to zero values.
 Note that if $\bfW$ is not positive-definite, then $\|\cdot\|_{\bfW}$ is semi-norm and the problem \eqref{eq:wls_gen} may become ill-posed. In particular, the topology of $\overline{\calD_r}$ is not consistent with the semi-norm and
 therefore the minimum in \eqref{eq:wls_gen} should be changed to infimum, which can be not achieved at time series from $\overline{\calD_r}$.
 	
\textbf{Different approaches for solving  \eqref{eq:wls}.} 
The optimization problem \eqref{eq:wls} is  non-convex  with many local minima \cite{Ottaviani2014lraexact}.
The problem \eqref{eq:wls} is commonly considered as a structured (more precisely, Hankel) low-rank approximation problem (SLRA, HSLRA) \cite{Chu2003,Markovsky2006,Markovsky2019}.
A well-known subspace-based method for solving \eqref{eq:wls}
is called `Cadzow iterations' \cite{Cadzow1988} and belongs to the class
of alternating-projection methods.
The method of Cadzow iterations can be extended to a class of oblique Cadzow iterations in the norm,
which differs from the Euclidean norm \cite{Gillard2016}.
The method has two drawbacks: first, the properties of the limiting point of the Cadzow iterations are unknown \cite{Andersson2013}
and second, it tries to solve the problem \eqref{eq:wls} with a weight matrix which generally differs
from the given $\Sigminus$.
Therefore, it is not optimal (the method does not provide the MLE), even for the case of white Gaussian noise \cite{DeMoor1994}.
The reason is that the problems are commonly stated in SLRA as matrix approximation problems, while the original problem
\eqref{eq:wls} is stated in terms of time series.

Many methods have been proposed to solve  HSLRA, including the Riemannian SVD \cite{DeMoor1994}, Structured total least-norm \cite{Lemmerling2000}, Newton-like iterations \cite{Schost2014}, proximal iterations \cite{Condat.Hirabayashi2015}, symbolic computations \cite{Ottaviani2014lraexact}, stochastic optimization \cite{Gillard2013}, fixed point iterations \cite{Andersson2013}, a penalization approach \cite{Ishteva.etal2014}.

Since we consider the problem of WLS time series approximation, which generally differs from the problem of matrix approximation due to different weights (see e.g. \cite{Zvonarev.Golyandina2017}), let us use as a benchmark the effective and general approach of Markovsky and Usevich \cite{Usevich2012,Usevich2014}, which is based on the variable projection principle \cite{Golub.Pereyr2003} combined with the Gauss-Newton method for solving the arising optimization subproblem. The method from \cite{Usevich2012,Usevich2014} is able to deal with the problem in the form \eqref{eq:wls}, i.e., exactly with the given weight matrix; moreover, it is elaborated in general form for a wide class of structured matrices and at the same time its iteration complexity scales linearly with the length of data for a class of weight matrices. Thus, the method can be considered as a start-of-art method of low-rank time series approximation.
Nevertheless, the approach has a couple of disadvantages. First, the Cholesky factorization is used for solving least-squares subproblems to obtain a fast algorithm; unfortunately, this squares the condition number (more stable  decompositions like QR factorization are slower). Then, the method is efficient only if the inverse of the weight matrix is banded.
Note that the approach of Markovsky and Usevich can be applied to the case of rank-deficient matrices $\bfW$ in \cite{Markovsky2013missing} and \cite[Section 4.4]{Markovsky2019}. However, it is not clear how to implement the proposed algorithm effectively from the viewpoint of computational cost.

\textbf{The proposed approach.} Let us consider another approach to solving the problem \eqref{eq:wls}, which is similar but different.
It is discussed in Section~\ref{sec:lrr} that each time series $\tsS$ from $\overline\calD_r$ is characterized by a vector $\bfa\in \spR^{r+1}$, which provides the coefficients of a generalized linear recurrence relation (GLRR) governing the time series, i.e. $\bfa^\rmT \calT_{r+1}(\tsS)$ is the zero vector. For each $\bfa$, we can consider the space $\calZ(\bfa)$ of signals governed by the GLRR with the given coefficients. Algorithms that use the variable projection method for solving the problem \eqref{eq:wls} include the projection to $\calZ(\bfa)$ as a subproblem. We call $\calZ(\bfa)$ the image space and its orthogonal complement $\calQ(\bfa)$ the left-null space, see Section~\ref{subsec:subspace_approach} for notation.

In this paper, we propose to overcome the drawbacks of the method from \cite{Usevich2012,Usevich2014} in the following manner. First, we consider a modified Gauss-Newton iteration method by using a special parameterization of the problem; this modification helps to avoid computing the pseudoinverse of the Jacobian matrix (compare \eqref{eq:gauss_simple} and \eqref{eq:iterGNfinal}). Then, unlike \cite{Usevich2014}, the projection is calculated directly on the image space $\calZ(\bfa)$ and is not obtained through projecting on the space $\calQ(\bfa)$. This feature of the proposed algorithm is emphasized in the paper title. Finally, for calculating the projection, we use fast algorithms with improved numerical stability (the compensated Horner scheme, see Section~\ref{subsec:hornerscheme}).
As a result, the proposed method can be just slightly slower and is much faster in many real-life scenarios, but also is more stable (see Section~\ref{sec:comparison} with the comparison results).
Moreover, the algorithm with direct projections to the image space $\calZ(\bfa)$ can be extended to the case of a degenerate weight matrix $\bfW$
 (in particular, to the case of missing values) without loss of effectiveness (see Remark~\ref{rem:wlsseminorm} to the algorithms); compare with that in \cite{Markovsky2013missing}, where projections to the subspace $\calQ(\bfa)$ are used and the computational cost considerably increases for degenerate weight matrices.

We also study some other properties of the problem including properties of $\calD_r$ in the considered parameterization.
The obtained results can be useful beyond the scope of this paper. In particular,
the induced parametric form of the tangent subspace at a given point of $\calD_r$ can be useful for investigating the local properties of the problem solution. Also, the effective algorithm for calculating the projection to the subspace of series governed by a specific linear recurrence relation,
which is proposed in Section~\ref{sec:ZofA},
can be used in different algorithms within HSLRA.

\textbf{Comments to general terminology}. 
Let us explain the terminology, which we use. For each window length $L$, there is a one-to-one correspondence between a time series $\tsX_N$ of length $N$ and its $L$-trajectory matrix $\calT_L(\tsX_N)$ (we use this name taken from singular spectrum analysis) in $\spR^{L\times (N-L+1)}$. Different window lengths $L$ correspond to different (unweighted) matrix approximations. Therefore, the low-rank matrix approximations can be varied for different $L$. The notion of low-rank signals does not depend on $L$ and therefore is not related to matrices (generally speaking). Moreover, the solved problem \eqref{eq:wls} is stated in terms of time series, not in terms of matrices. For approximation by low-rank signals, $N$ weights are set for time series points, not for matrix entries. That is why we use the notion ``low-rank signals''.

Note that to have equivalent optimization problems for the matrix SLRA itself and for time series (vector) LRA, one should consider weighted versions and care about the correspondence of weights.
In \cite{Zvonarev.Golyandina2017} (see also a general description in \cite[Section 3.4]{Golyandina.etal2018}), the problem is solved as a matrix approximation problem with appropriate weights. In this paper, we consider the problem of time series low-rank approximation \eqref{eq:wls}.

\textbf{Structure of the paper}.
In  Section~\ref{sec:parametrization} we consider a parameterization of $\calD_r$ and its properties, which help to
construct effective algorithms.
In Section~\ref{sec:optim} we describe the known (VPGN) and the new proposed (MGN) iterative methods
for solving the optimization problem~\eqref{eq:wls}. The algorithm VPGN is described in the way different from
that in \cite{Usevich2014}, since the description in \cite{Usevich2014} is performed for general SLRA problems
and therefore it is difficult to apply it to the particular case of Hankel SLRA for time series.
(For the convenience of readers, we include Table~\ref{table:defines} containing equivalent notations.)
In Section~\ref{sec:ZofA} we propose effective algorithms for implementing the key steps of the main algorithms.
Section~\ref{sec:MGNand VPGN} presents the algorithms with the implementations of VPGN and MGN.
In Section~\ref{sec:comparison} we compare computational costs and numerical stability of the VPGN and MGN algorithms.
Section~\ref{sec:conclusion} concludes the paper.
Long proofs and technical details are relegated to the appendix.

\textbf{Main notation}.
In this paper, we use lowercase letters ($a$,$b$,\ldots) and also $L$, $K$, $M$, $N$ for scalars, bold lowercase letters ($\bfa$,$\bfb$,\ldots) for vectors, bold uppercase letters ($\bfA$,$\bfB$,\ldots) for matrices, and
the calligraphic font for sets. Formally, time series are vectors; however, we use the uppercase sans serif font ($\tsA$,$\tsB$,\ldots) for time series to distinguish them from ordinary vectors.
Additionally, $\bfI_{M}\in \spR^{M\times M}$ is the identity matrix, $\bm{0}_{M \times k}$ denotes the $M \times k$ zero matrix, $\bm{0}_M$ denotes the zero vector in $\spR^M$,
$\bfe_i$ is the $i$-th standard basis vector.

Denote $\bfb_{\vecrow{\calC}}$ the vector consisting of
the elements of a vector $\bfb$ with the numbers from a set $\calC$,
For matrices, denote $\bfB_{\row{\calC}}$ the matrix consisting
of rows of a matrix $\bfB$ with the numbers from $\calC$ and $\bfB_{\col{\calC}}$
the matrix consisting
of columns of a matrix $\bfB$ with the numbers from $\calC$.

Finally, we put a brief list of main common symbols and acronyms.\\
LRR is linear recurrence relation.\\
GLRR($\bfa$) is generalized LRR with the coefficients given by $\bfa$.\\
$\calD_r$ is the set of time series of rank $r$.\\
$\overline{\calD_r}$ is the set of time series of rank not larger than $r$.\\
$\calZ(\bfa) \in \spR^N$ is the set of time series of length $N$ governed by the minimal GLRR($\bfa$);
$\bfZ(\bfa)$ is the matrix consisting of its basis vectors.\\
$\calQ(\bfa)$ is the orthogonal complement to $\calZ(\bfa)$; $\bfQ(\bfa)$ is the matrix consisting of its special basis vectors in the form \eqref{op:Q}.\\
$\bfW\in \spR^{N\times N}$ is a weight matrix.\\
$\winverse{\bfF}{\bfW}$ is the weighted pseudoinverse matrix; $\inverse{\bfF}$ stands for $\winverse{\bfF}{\bfI_N}$.\\
$\bfJ_{S}$ is the Jacobian matrix of a map $S$.\\
$\calT_M: \spR^N \rightarrow \spR^{M\times (N-M+1)}$ is the embedding operator, which constructs the $M$-trajectory matrix.\\
$\fullop$: $\spR^{r} \to \spR^{r+1}$ is the operator, which inserts $-1$ at the $\tau$ position.\\
$\Proj_{\calL,\bfW}$ is the $\bfW$-orthogonal projection onto $\calL$, $\Proj_{\bfL,\bfW}$ is the $\bfW$-orthogonal projection onto $\colspace(\bfL)$; if $\bfW$ is the identity matrix, it is omitted in the notation.\\
$S_{\tau}^\star(\Ai0) = \Proj_{\calZ(\fullop(\Ai0)), \bfW}(\tsX)$, where $\Ai0\in \spR^r$.

	\section{Parameterization of low-rank series}
    \label{sec:parametrization}
    \subsection{Linear recurrence relations}
    \label{sec:lrr}

    It is well known \cite[Theorem 3.1.1]{Hall1998} that
    a time series of the form \eqref{eq:model} satisfies a linear recurrence relation (LRR) of some order $m$:
    \begin{equation}
    \label{eq:lrr}
    s_n = \sum_{k=1}^m b_k s_{n-k}, n = m+1, \ldots, N; b_m\neq 0.
    \end{equation}
    One time series can be governed by many different LRRs. The LRR of minimal order $r$ (it is unique) is called minimal.
    The corresponding time series has rank $r$. The minimal LRR uniquely defines the form of \eqref{eq:model} and the parameters $m_k$, $\alpha_k$, $\omega_k$.

    The relations \eqref{eq:lrr} can be expressed in vector form as
    $\bfa^\rmT \calT_{m+1}(\tsS) = \bfzero_{N-m}^\rmT$, where the vector $\bfa = (b_m, \ldots, b_1, -1)^\rmT \in \spR^{m+1}$.
    The vector $\bfa$ corresponding to the minimal LRR ($m=r+1$) and the first $r$  values of the series $\tsS$ uniquely determine the whole series  $\tsS$.
    Therefore, $r$ coefficients of an LRR of order $r$ and $r$ initial values
    ($2r$ parameters altogether) can be chosen as parameters of a series of rank $r$.
    However, this parameterization does not describe the whole set $\calD_r$ \cite[Theorem 5.1]{Golyandina.etal2001}.

Let us generalize LRRs.
We say that a time series satisfies a generalized LRR (GLRR) of order $m$ if $\bfa^\rmT \calT_{m+1}(\tsS) = \bfzero_{N-m}^\rmT$ for some non-zero $\bfa \in \spR^{m+1}$;
we call this linear relation GLRR($\bfa$).
As well as for LRRs, the minimal GLRR can be introduced.
The difference between a GLRR  and an ordinary LRR is that the last coefficient in the GLRR is not necessarily non-zero and
therefore the GLRR does not necessarily set a recurrence.
However, at least one of the coefficients of the GLRR should be non-zero.
GLRRs correspond exactly to the first characteristic polynomial in \cite[Definition 5.4]{Heinig1984}.

Let us demonstrate the difference between LRR and GLRR by an example. Let $\tsS = (s_1,\ldots,s_N)^\rmT$ be a signal and $\bfa = (a_1, a_2, a_3)^\rmT$.
Then GLRR($\bfa$) and LRR($\bfa$) mean the same:
$a_1 s_i + a_2 s_{i+1} + a_3 s_{i+2} = 0$ for $i=1,\ldots,N-2$. For LRR($\bfa$), we state that $a_3  = -1$ (or just not equal to 0). Then this linear relation becomes a recurrence relation since $s_{i+2} = a_1 s_i + a_2 s_{i+1}$.
For GLRR($\bfa$), we assume that some of $a_i$ is not zero (or equal to $-1$). It may be $a_1$ or $a_2$ or $a_3$.

Any signal of rank $r$ satisfies a GLRR($\bfa$), where $\bfa\in \spR^{r+1}$. However, not each signal of rank $r$ corresponds to an LRR. E.g., $\tsS=(1,1,1,1,1,2)^\rmT$ has rank 2 and does not satisfy an LRR. However, it satisfies the GLRR($\bfa$) with $\bfa = (1,-1,0)^\rmT$. Therefore, we consider the parameterization with the help of GLRR($\bfa$). In fact, the same approach is used in \cite{Usevich2012, Usevich2014}. It is indicated in Table~\ref{table:defines} that $\bfa$ in this paper corresponds to $R$ in \cite{Usevich2012, Usevich2014}.

The following properties clarify the structure of the spaces $\calD_r$ and $\overline{\calD_r}$:
(a)
$\overline{\calD_r} = \{\tsY: \exists \bfa\in \spR^{r+1}, \bfa\neq \bfzero_{r+1}: \bfa^\rmT \calT_{r+1}(\tsS) = \bfzero_{N-r}^\rmT\}$ or, equivalently,
$\tsY \in \overline{\calD_r}$ if and only if there exists a GLRR($\bfa$) of order $r$, which governs $\tsY$;
(b)
$\tsY \in \calD_r$ if and only if there exists a GLRR($\bfa$) of order $r$, which governs $\tsY$, and this GLRR is minimal.

\subsection{Subspace approach}
\label{subsec:subspace_approach}
	Let  $\calZ(\bfa)$, $\bfa\in \spR^{r+1}$, be the space of time series of length $N$ governed by the GLRR($\bfa$);
that is, $\calZ(\bfa) = \{\tsS: \bfa^\rmT \calT_{r+1}(\tsS) = \bfzero_{N-r}^\rmT \}$.
Therefore $\overline{\calD_r} = \bigcup \limits_{\bfa} \calZ(\bfa)$.

Let  $\bfQ^{M, d}$ be the operator $\spR^{d+1} \to \spR^{M \times (M - d)}$, which is defined  by
\begin{equation}\label{op:Q}
\big(\bfQ^{M,d}(\bfb)\big)^\mathrm{T} = \begin{pmatrix}
b_1 & b_2 & \dots & \dots & b_{d+1} & 0 & \dots & 0 \\
0 & b_1 & b_2 & \dots &  \dots & b_{d+1} & \ddots & \vdots \\
\vdots & \ddots  & \ddots & \ddots & \ddots & \ddots & \ddots & 0 \\
0 & \dots & 0 & b_1 & b_2 & \ddots & \ddots & b_{d+1} \\
\end{pmatrix},
	\end{equation}
where $\bfb=(b_1, \ldots, b_{d+1})^\rmT \in \spR^{d+1}$.
Then the other convenient form of $\calZ(\bfa)$ is $\calZ(\bfa) = \{\tsS: \bfQ^\rmT(\bfa) \tsS = \bfzero_{N-r} \}$,
where $\bfQ = \bfQ^{N,r}$. 

The following notation will be used below: $\calQ(\bfa) = \colspace(\bfQ(\bfa))$ and denote $\bfZ(\bfa)$ a matrix whose column vectors form a basis of $\calZ(\bfa)$. The space $\calZ(\bfa)$ is the image space of $\bfZ(\bfa)$ for any choice of the basis, while $\calQ(\bfa)$ is the left-null space of $\bfZ(\bfa)$ (or, the same, the kernel of $\bfZ(\bfa)^\rmT$); $\calQ(\bfa)$ is the orthogonal complement to $\calZ(\bfa)$.

\subsection{Parameterization}
\label{sec:param}
Consider a series $\tsS_0\in \calD_r$, which satisfies a minimal GLRR($\bfa_0$) of order $r$ defined by a non-zero vector $\bfa_0 = (a_1^{(0)}, \ldots, $ $a_{r+1}^{(0)})^\rmT$.
Let us fix $\i0$ such that $a^{(0)}_{\i0} \neq 0$. Since GLRR($\bfa_0$) is invariant to multiplication by a constant,
we assume that $a^{(0)}_{\i0} = -1$. This condition on $\i0$ is considered to be valid hereinafter.
Let us build a parameterization of $\calD_r$ in the vicinity of $\tsS_0$; parameterization depends on the index $\i0$.
Note that we can not construct a global parameterization, since for different points of $\calD_r$ the index $\i0$,
which corresponds to a non-zero element of $\bfa_0$, can differ.

In the case of a series governed by an ordinary LRR($\bfa$), $\bfa\in \spR^{r+1}$, since the last coordinate of $\bfa$ is equal to $-1$,
the series is uniquely determined by the first $r$ elements of $\bfa$ and $r$  initial values of the series.
Then, applying the LRR to the initial data, which are taken from the series that is governed by the LRR,
we restore this series.

In the case of an arbitrary series from $\calD_r$, the approach is similar but a bit more complicated. For example,
we should take the boundary data ($\i0-1$ values at the beginning, and $r+1-\i0$ values at the end) instead of the $r$ initial values at the beginning of the series; also, the GLRR is not in fact recurrent (we keep notation
to show that LRRs are a particular case of GLRRs).

Denote $\calI(\i0) = \{1,\ldots, N\} \setminus \{\i0,\ldots, N-r-1+\i0\}$ and
$\calK(\i0) = \{1,\ldots,r+1\} \setminus \{\i0\}$ two sets of size $r$.
The set $\calI(\i0)$ consists of the numbers of series values (we call them boundary data), which
are enough to find all the series values with the help of $\bfa$ (more precisely,
by elements of $\bfa$ with numbers from $\calK(\i0)$).
Then $\bfa_{\vecrow{\calK({\i0})} }\in \spR^{r}$
defines the vector consisting of the elements of a vector $\bfa \in \spR^{r+1}$
with the numbers from $\calK({\i0})$.

To simplify notation, let us introduce the operator $\fullop$: $\spR^{r} \to \spR^{r+1}$, which acts as follows. Let $\Ai0\in \spR^{r}$ and $\fullop(\Ai0) = \bfa$. Then $\bfa = (a_1, \ldots, $ $a_{r+1})^\rmT$ is such that $\bfa_{\vecrow{\calK({\i0})}} = \Ai0$ and $a_{\i0} = -1$; that is, $\Ai0 \in \spR^r$ is extended to $\bfa \in \spR^{r+1}$ by inserting $-1$ at the $\i0$-th position.
In this notation, $\bfa_{\vecrow{\calK({\i0})}} = \fullop^{-1}(\bfa)$.

Theorem~\ref{th:parametrization} defines the parameterization, which will be used in what follows.
The explicit form of this parameterization is given in Proposition~\ref{prop:parametrization}.

\begin{theorem} \label{th:parametrization}
	Let $\bfa_0\in \spR^{r+1}$, $a_\i0^{(0)} = -1$, and $\tsS_0 \in \calD_r$ satisfy the GLRR($\bfa_0$). Then there exists a unique one-to-one mapping $S_{\tau}: \spR^{2r} \to \calD_r$ between a neighborhood of the point $\left((\tsS_0)_{\calI(\i0)}, (\bfa_0)_{\calK(\i0)}\right)^\rmT \in \spR^{2r}$ and the intersection of a neighborhood of $\tsS_0$ with the set $\calD_r$, which satisfies the following relations:
for $\tsS = S_{\tau}(\Si0, \Ai0)$, where $\Si0, \Ai0 \in \spR^r$, we have
\begin{itemize}
\item
 $(\tsS)_{\vecrow{\calI(\i0)}} = \Si0$;
 \item
 $\tsS \in \calD_r$ is governed by the GLRR($\fullop(\Ai0)$).
 \end{itemize}
\end{theorem}

\begin{proposition} \label{prop:parametrization}
Let $\bfa_0\in \spR^{r+1}$, $a_\i0^{(0)} = -1$, and $\bfZ_0 \in \spR^{N \times r}$ consist of basis vectors of $\calZ(\bfa_0)$. Consider the parameterizing mapping $S_\tau$, introduced in Theorem~\ref{th:parametrization}.\\
1. Let $(\Si0, \Ai0)^\rmT \in \spR^{2r}$ and denote $\bfa = \fullop(\Ai0)$.
 Denote $\Proj_{\calZ(\bfa)}$ the orthogonal projection onto $\calZ(\bfa)$. Then for $\bfZ = \Proj_{\calZ(\bfa)} \bfZ_0\in \spR^{N\times r}$ and $\bfG = \bfZ \left(\bfZ_{\row{\calI({\i0})}}\right)^{-1}$, where $\bfZ_{\row{\calI({\i0})}}\in \spR^{r\times r}$, the mapping $S_\tau$ has the explicit form
		\begin{equation}\label{eq:param}
		\tsS = S_\tau(\Si0, \Ai0) = \bfG \Si0.
		\end{equation}
2. The inverse of the mapping $S_\tau$ is given as follows. Let $\tsS = S_\tau(\Si0, \Ai0)$. Then
\begin{equation}
         \label{eq:param_rev}
        \Si0 = (\tsS)_{\vecrow{\calI(\i0)}},\qquad
		\Ai0 = (-\hat{\bfa}/\hat{a}_\i0)_{\vecrow{\calK(\i0)}},
\end{equation}
		where $\hat{\bfa} = \hat{\bfa}(\tsS) = (\hat a_1, \ldots, \hat a_{r+1})^\rmT = \left(\bfI_{r+1} -  \Proj_{\calL(\tsS)}\right) \bfa_0$, $\calL(\tsS) = \colspace\left(\calT_{r+1}(\tsS)\right)$, $\Proj_{\calL(\tsS)}$ is the orthogonal projection onto $\calL(\tsS)$.
\end{proposition}

\begin{proof}
See the proof of Theorem~\ref{th:parametrization} together with Proposition~\ref{prop:parametrization} in Section~\ref{sec:th:parametrization}.
\end{proof}

Note that for different series $\tsS_0\in \calD_r$ we have different parameterizations of $\calD_r$ in vicinities of $\tsS_0$.
Moreover, for a fixed $\tsS_0$, there is a variety of parameterizations provided by different choices of the index $\i0$.

\subsection{Smoothness of parameterization and derivatives}
\begin{theorem}
	\label{th:param_smooth}
Let $\bfa_0\in \spR^{r+1}$, $a_\i0^{(0)} = -1$, and $\tsS_0 \in \calD_r$ satisfy the GLRR($\bfa_0$). Then the parameterization $S_{\tau}(\Si0, \Ai0)$, which is introduced in Theorem~\ref{th:parametrization} and Proposition~\ref{prop:parametrization},
 is a smooth diffeomorphism
 between a neighborhood of the point $\left((\tsS_0)_{\calI(\i0)}, (\bfa_0)_{\calK(\i0)}\right)^\rmT \in \spR^{2r}$ and the intersection of a neighborhood of $\tsS_0$ with the set $\calD_r$.
\end{theorem}
\begin{proof}
See the proof in Section~\ref{sec:th:param_smooth}.
\end{proof}

Let us consider the derivatives of the parameterizing mapping.
Let the series $\tsS$ belong to a sufficient small neighborhood of $\tsS_0$ and be parameterized as $\tsS = S_{\tau}(\Si0, \Ai0)$.
Denote $\bfJ_{S_{\tau}} = \bfJ_{S_{\tau}}(\Si0, \Ai0) \in \spR^{N\times 2r}$ the Jacobian matrix of  $S_{\tau}(\Si0, \Ai0)$.

By definition, the tangent subspace at the point $\tsS$ coincides with $\colspace\left(\bfJ_{S_{\tau}}(\Si0, \Ai0)\right)$.
Note that the tangent subspace is invariant with respect to the choice of a certain parameterization of $\calD_r$ in the vicinity of $\tsS$.

Define by $\bfa^2$ the acyclic convolution of $\bfa$ with itself:
\begin{equation*}
	\bfa^2 = (a^{(2)}_i) \in \spR^{2r+1}, \quad a^{(2)}_i = \sum_{j=\max(1, i - r)}^{\min(i, r+1)} a_j a_{i - j + 1}.
\end{equation*}
		
\begin{theorem}
\label{th:tangent}
	The tangent subspace to $\calD_r$ at the point $\tsS$ has dimension $2r$ and is equal to $\calZ(\bfa^2)$.
\end{theorem}
\begin{proof}
See the proof in Section~\ref{sec:th:tangent}.
\end{proof}

\section{Optimization}
\label{sec:optim}
Let us consider different numerical methods for solving the problem \eqref{eq:wls}.
First, note that we  search for a local minimum. Then, since the objective function is smooth in the considered
parameterization, one can apply the conventional weighted version of the Gauss-Newton method (GN), see \cite{nocedal2006numerical} for details.
However, this approach appears to be numerically unstable and has a high computational cost.

In \cite{Usevich2014}, the variable-projection method (VP) is used
for solving the minimization problem. When the reduced minimization problem
is solved again by the Gauss-Newton method, we will refer to it as VPGN.

We propose a similar (but different) approach called Modified Gauss-Newton method (MGN),
which appears to have some  advantages in comparison with VPGN that is one of the best methods
for solving the problem \eqref{eq:wls}. Below we will show that the MGN algorithm consists of different numerical sub-problems to be solved, which are more well-conditioned  than in the VPGN case (see \cite{Deuflhard.Hohmann2003}, where different properties of problems such as stability and well-conditioning are discussed); thereby, MGN allows a better numerically stable implementation.

The structure of this section is as follows.
After a brief discussion of the problem \eqref{eq:wls} we start with the description of the methods GN and VP for a general optimization problem; then we apply these methods to \eqref{eq:wls} and finally present the new method MGN.

Note that the considered methods are used for solving a weighted least-squares problem and therefore
we consider their weighted versions, omitting `weighted' in the names of the methods.

Let us introduce notation, which is used in this section.
For some matrix $\bfF = \spR^{N\times p}$, define its weighted pseudoinverse \cite{Stewart1989}
$\winverse{\bfF}{\bfW} = (\bfF^\rmT \bfW \bfF)^{-1}\bfF^\rmT \bfW$;
this pseudoinverse arises in the solution of the linear weighted least-squares problem $\min_\bfp \|\bfy - \bfF \bfp\|_{\bfW}^2$
with $\bfy \in \spR^{N}$, since its solution is equal to $\bfp_{\mathrm{min}} = \winverse{\bfF}{\bfW} \bfy$. In the particular case $\bfW = \bfI_N$, $\winverse{\bfF}{\bfW}$ is the ordinary pseudoinverse; we will denote it $\inverse{\bfF}$.
Denote the projection (it is oblique if $\bfW$ is not the identity matrix) onto the column space $\calF$ of a matrix $\bfF$ as $\Proj_{\bfF, \bfW} = \bfF \winverse{\bfF}{\bfW}$.
If it is not important which particular basis of $\calF$ is considered, we use the notation $\Proj_{\calF, \bfW}$.

\begin{remark}
\label{rem:degenerate}
Let us consider a degenerate case when $\bfF^\rmT \bfW \bfF$ is not positive definite or, the same,
$\bfW ^{1/2}\bfF$ is rank-deficient ($\bfW ^{1/2}$ is the principal square root of $\bfW$). Then we can use a different representation for the weighted pseudoinverse:
$\winverse{\bfF}{\bfW} = \inverse{(\bfW^{1/2} \bfF)} \bfW^{1/2}$.
This corresponds to the minimum-(semi)norm solution of the corresponding WLS problem $\min_\bfp \|\bfy - \bfF \bfp\|_{\bfW}^2$.
Although,
the projection $\Proj_{\bfF, \bfW}$ is generally not uniquely defined in the degenerate case, we will consider its uniquely defined version given by the formula  $\Proj_{\bfF, \bfW} = \bfF \winverse{\bfF}{\bfW}$.

The matrix $\bfW ^{1/2}\bfF$ is rank-deficient if $\bfF$ is rank-deficient.
However, for full-rank $\bfF$ and degenerate $\bfW$, $\bfW ^{1/2}\bfF$ is not necessarily rank-deficient.
For example, if the orthogonal projections of the columns of $\bfF$ on $\colspace(\bfW)$ are linearly independent,
then  $\bfW ^{1/2}\bfF$ is full-rank.
\end{remark}

\begin{remark}
\label{rem:complexF}
If the matrix $\bfF$ is complex, the above formulas and considerations are still valid with the change of the transpose $\bfF^\rmT$ to the complex conjugate $\bfF^*$.
\end{remark}

\subsection{Properties of the optimization problem \eqref{eq:wls}}
 The following lemma shows that the global minimum of \eqref{eq:wls} belongs to $\calD_r$ for the majority of $\tsX$. Therefore, it is sufficient
 to find the minimum in the set of series of exact rank $r$.

 \begin{lemma}\label{lemma:minindr}
	Let $\tsX \notin \overline{\calD_r} \setminus \calD_r$ and $\bfW$ be positive definite. Then any point of the global minima in the problem \eqref{eq:wls} belongs to $\calD_r$.
\end{lemma}
\begin{proof}
See the proof in Section~\ref{sec:lemma:minindr}.
\end{proof}

Thus, the problem \eqref{eq:wls} can be considered as a minimization problem in $\calD_r$; therefore, in the chosen parameterization of $\calD_r$ (see Section~\ref{sec:param}), the problem \eqref{eq:wls} in the vicinity of $\tsS_0$ has the form
	\begin{equation}
    \label{eq:wlsP}
	\bfp^\star = \argmin_{\bfp} \| \tsX - S(\bfp) \|_{\bfW},
	\end{equation}
where $\bfp=(\Si0, \Ai0)$, $S = S_{\tau}$.
Since $S(\bfp)$ is a differentiable function of $\bfp$ due to Theorem~\ref{th:param_smooth}
for an appropriate choice of $\i0$, numerical methods like the Gauss-Newton method can be applied to the solution of \eqref{eq:wlsP}.

The following theorem helps to detect if the found solution is a local minimum. Recall that $\calZ(\bfa^2)$ determines the
tangent subspace (Theorem~\ref{th:tangent}).

\begin{lemma}[Necessary conditions for local minima]\label{lemma:locminnec}
	Let $\bfW$ be positive definite. If the series $\tsX_0 \in \calD_r$ which is governed by a GLRR($\bfa_0$) provides a local minimum in the problem \eqref{eq:wls},
then $\Proj_{\calZ(\bfa_0^2), \Sigminus}(\tsX - \tsX_0) = \bfzero_N$.
\end{lemma}
\begin{proof}

	Let us take an appropriate index $\i0$ together with the parameterization $S_{\tau}(\Si0, \Ai0)$ introduced in Theorem \ref{th:parametrization}.
Due to Theorem \ref{th:param_smooth}, the objective function $\|\tsX -  S_{\tau}(\Si0, \Ai0)\|^2_\bfW$ is smooth in the vicinity of $\left((\bfs_0)_{\calI(\i0)}, (\bfa_0)_{\calK(\i0)}\right)^\rmT \in \spR^{2r}$. Theorem~\ref{th:tangent} together with \cite[Theorem 2.2]{nocedal2006numerical}, which formulates the necessary conditions for a minimum in a general case, applied to the considered
objective function finish the proof.
\end{proof}

Note that Lemma~\ref{lemma:locminnec} provides the necessary condition only. According to \cite[Theorem 2.3]{nocedal2006numerical},
 sufficient conditions for a minimum include positive definiteness of the Hessian of the objective function. For the present, we can check this positive definiteness only numerically.

\subsubsection{The case of ill-posed problem}
\label{rem:seminorm}
	Up to this point, we assumed that the weight matrix $\bfW$ is full-rank.
 As we have mentioned, the problem \eqref{eq:wls} for a degenerate weight matrix $\bfW$ can be ill-posed, since the set $\overline \calD_r$ becomes not closed. This means that there are time series $\tsX$ such that the infimum of the objective function is not achieved at $\overline \calD_r$ and therefore the problem \eqref{eq:wls} cannot be solved.

 Let us demonstrate this by an example.
 Let $N \ge 3$ and take the series $\tsX = \tsX_{N} = (0, \ldots, 0, 1, 0)^\rmT \in \spR^N$. Consider
 the simple case $r=1$ and the weight matrix $\bfW = \diag((1, \ldots, 1, 0)^\rmT)$, which is evidently degenerate.
  Then for exponential time series $\tsY(\mu) = (1, \mu, \mu^2, \ldots, \mu^{N-1})^\rmT / \mu^{N-2}$ we have  $\|\tsY(\mu) - \tsX\|_{\bfW} \to 0$ as $\mu \to \infty$. However, there does not exist a series $\widetilde \tsY \in \overline \calD_1$ such that $\| \tsX - \widetilde \tsY \|_\bfW = 0$. Indeed, if we suppose that $\| \tsX - \widetilde \tsY \|_\bfW = 0$, then $\widetilde \tsY$ has the form $\widetilde \tsY = (0, \ldots, 0, 1, \tilde y)^\rmT$, where $\tilde y \in \spR$. This is a contradiction, since $\widetilde \tsY$ is not governed by a GLRR($\bfa$) of order $1$ for any $\tilde y$ (a non-zero $\bfa$ should be orthogonal to both $(0, 1)^\rmT$ and $(1, \tilde y)^\rmT$).

\subsection{Methods for solving a general nonlinear least-squares problem}\label{sec:opt_gen}
	Let $\bfx \in \spR^N$ be a given vector and consider a general WLS minimization problem
	\begin{equation}
    \label{eq:gen_optim}
	\bfp^\star = \argmin_\bfp \|\bfx - S(\bfp)\|_{\bfW}^2,
	\end{equation}
	where $\bfp \in \spR^p$ is the vector of parameters, $S: \spR^p \to \spR^N$ is some parameterization of a subset of $\spR^N$ such that $S(\bfp)$ is a differentiable vector-function of $\bfp$,
$\bfW \in \spR^{N\times N}$ is
   a positive (semi-)definite symmetric matrix.

   If the problem \eqref{eq:gen_optim} is non-linear, iterative methods with linearization
at each iteration are commonly used, such as the Gauss-Newton method or its variations \cite{nocedal2006numerical}.
One of the commonly used variations is the Levenberg-Marquardt method,
which is a regularized version of the Gauss-Newton method. This regularization improves the method
far from the minimum and does not affect near the minimum. Therefore, in the paper, we consider the Gauss-Newton method without regularization.
We use a weighted Gauss-Newton method, which is a straightforward extension of the unweighted version.

\subsubsection{Gauss-Newton method} \label{sec:optim_wgn}

One iteration of the Gauss-Newton algorithm  with step $\gamma$ is
\begin{equation}
\label{eq:GN_P}
\bfp_{k+1} = \bfp_{k} + \gamma \winverse{\bfJ_{S}(\bfp_k)}{\bfW} (\bfx - S(\bfp_k)),
\end{equation}
where $\bfJ_{S}(\bfp_k)$ is the Jacobian matrix of $S(\bfp)$ at $\bfp_k$. Note that the iteration step \eqref{eq:GN_P} is uniquely defined for any positive semi-definite matrix, see Remark~\ref{rem:degenerate}.
The choice of step $\gamma$ is a separate problem. For example, one can apply the backtracking line search starting at $\gamma = 1$
and then decreasing the step if the next value is worse (that is, if the value of the objective functional increases).


An additional aim of the WLS problem is to find the approximation $S(\bfp^\star)$ of $\bfx$, where $\bfp^\star$ is the solution of \eqref{eq:gen_optim}.
 Then we can write \eqref{eq:GN_P} in the form of iterations of approximations:
\begin{equation} \label{eq:gnwls}
S(\bfp_{k+1}) = S \left(\bfp_{k} + \gamma \winverse{\bfJ_{S}(\bfp_k)}{\bfW} \left(\bfx - S(\bfp_k)\right)\right).
\end{equation}

The following remark explains the approach, which underlies the Modified Gauss-Newton method proposed in this paper.
\begin{remark} \label{rem:replacement}
	The iteration step \eqref{eq:gnwls} can be changed by means of the change of $S(\bfp_{k+1})$ to $\widetilde S(\bfp_{k+1})$, where $\widetilde S(\bfp_{k+1})$ is such that  $\|\bfx - \widetilde S(\bfp_{k+1})\|_\bfW \le \|\bfx - S(\bfp_{k+1})\|_\bfW$. This trick is reasonable if $\widetilde S(\bfp_{k+1})$ can be calculated
faster and/or in a more stable way than $S(\bfp_{k+1})$.
\end{remark}

\subsubsection{Variable projection}
	Let $\bfp = \left(\begin{matrix}\bfb\\\bfc\end{matrix}\right)\in \spR^p$, $\bfb\in \spR^{p_1}$, $\bfc\in \spR^{p_2}$.
Consider the (weighted) least-squares problem \eqref{eq:gen_optim}, where $S(\bfp)$ is linear in $\bfc$ and the nonlinear part is defined by
$\bfG(\bfb)\in \spR^{N\times p_2}$:
\bea
S(\bfp) = \bfG(\bfb) \bfc.
\eea
This problem can be considered as a problem of projecting the data vector $\bfx$ onto a given set:
	\begin{equation}
\label{eq:full_optim}
	\min_{\bfy \in \calD} \|\bfx - \bfy \|_{\bfW}, \quad \text{where} \quad \calD =
    \Big\{\bfG(\bfb) \bfc \mid \left(\begin{matrix}\bfb\\\bfc\end{matrix}\right)\in \spR^p\Big\}.
	\end{equation}
	 Here $\{\varphi(z)\mid z\in \calC\}$ means the set of values of $\varphi(z)$ for $z\in \calC$.
    The variable projection method takes advantage of the known explicit solution of the subproblem:
	\begin{equation*}
	C^\star(\bfb) = \argmin_{\bfc} \|\bfx - \bfG(\bfb) \bfc \|_{\bfW} = \winverse{\bfG(\bfb)}{\bfW} \bfx.
	\end{equation*}

\smallskip	
	Denote $S^\star(\bfb) = \bfG(\bfb) C^\star(\bfb)$, $\calG(\bfb) = \{\bfG(\bfb)\bfc \mid \bfc\in \spR^{p_2} \}$. Then
	\begin{equation} \label{eq:minZ}
	S^\star(\bfb) = \argmin_{\bfs \in \calG(\bfb)} \|\bfx - \bfs \|_{\bfW}.
	\end{equation}
	Thus, we can reduce the problem \eqref{eq:full_optim} to the projection onto a subset $\calD^\star \subset \calD$ and
thereby to the optimization in the nonlinear part of parameters only:
	\begin{equation}\label{eq:vpprinciple}
	\min_{\bfy \in \calD^\star} \|\bfx - \bfy \|_{\bfW} \quad \text{with} \quad \calD^\star = \{S^\star(\bfb) \mid \bfb\in\spR^{p_1}\}.
	\end{equation}
This is called ``variable projection'' principle (see \cite{Golub.Pereyr2003} for the case of the Euclidean norm).

\subsection{Known iterative methods for solving the problem \eqref{eq:wls}}
Let us turn from a general nonlinear least-squares problem \eqref{eq:gen_optim} to the specific problem \eqref{eq:wls} in the form \eqref{eq:wlsP}.

A variation from the standard way of the use of iterative methods is that the parameterization $S_{\tau}(\bfp)$, $\bfp=(\Si0, \Ai0)$ (which is based on $\i0$)
is changed at each iteration in a particular way.
At $(k+1)$-th iteration, the parameterization is constructed in the vicinity of  $\bfa_0 = \bfa^{(k)}$. The index $\i0$, which determines the parameterization, is chosen in such a way to satisfy $a^{(0)}_\i0 \neq 0$.
We propose the following approach to the choice of $\i0$.
Let $\i0$ be the index of the maximum absolute entry of $\bfa_0$. Since the parameterization is invariant to the multiplication of $\bfa_0$ by a constant, it can be assumed that $a^{(0)}_\i0 = -1$ and $|a^{(0)}_i| \le 1$ for any $i$, $1 \le i \le r+1$.

\subsubsection{Weighted Gauss-Newton method for \eqref{eq:wls}}

The Gauss-Newton algorithm can be applied to the problem \eqref{eq:wlsP} in a straightforward manner, taking into consideration that the parameterization $S_{\tau}$ may be changed at each iteration.
The Gauss-Newton iteration has the form $\bfp_{k+1} = \bfp_k + \gamma \winverse{\bfJ_{S_{\tau}}(\bfp_k)}{\bfW} (\tsX - S_{\tau}(\bfp_k))$.

To apply the method, $S_{\tau}(\bfp_k)$ and the Jacobian matrix $\bfJ_{S_{\tau}}(\bfp_k)$
should be calculated. Formally, their computing can be implemented; however, the direct calculation is not
numerically stable and very time-consuming.

\subsubsection{Variable projection for \eqref{eq:wls} (VPGN)}
\label{sec:VPGN}
The explicit form of the parameterization $S_{\tau}(\bfp) = S_{\tau}(\Si0, \Ai0)$ given in \eqref{eq:param}, where $\Si0$ is presented in $S_{\tau}(\Si0, \Ai0)$ in a linear manner,
allows one to
apply the variable projection principle.

Assume that $\tsS_0$ is governed by a GLRR($\bfa_0$) with $a^{(0)}_{\i0} = -1$ and consider the problem \eqref{eq:wls} in the vicinity of the series $\tsS_0 \in \calD_r$.

Substitute in \eqref{eq:vpprinciple} $\calD = \overline \calD_r$, $\calD^\star = \calD_r^\star \subset \overline \calD_r$, where $\calD_r^\star = \{\Proj_{\calZ(\fullop(\Ai0)), \bfW}(\tsX) \mid \Ai0 \in \spR^{r}\}$, $\bfb = \Ai0$, $\bfG(\bfb) = \bfG$, where $\bfG = \bfZ \left(\bfZ_{\row{\calI({\i0})}}\right)^{-1}$ (see \eqref{eq:param}), $C^\star(\bfb) = \winverse{\bfG}{\bfW} \tsX$, $\bfG(\bfb)C^\star(\bfb) = \Proj_{\calZ(\fullop(\Ai0)), \bfW}(\tsX)\stackrel{\mathrm{def}}{=}S_{\tau}^\star(\Ai0)$.
Then we obtain the equivalent problem for projecting the elements from the set $\overline \calD_r$ to the subset $\calD_r^\star$,
where the parameter $\Si0$ is eliminated:
\begin{equation}\label{eq:wlsvp_set}
\tsY^\star = \argmin_{\tsY \in \calD_r^\star} \| \tsX - \tsY \|_{\bfW}.
\end{equation}
Therefore, we can present the problem \eqref{eq:wlsvp_set} in terms of the parameter $\Ai0$ only:
\begin{equation}\label{eq:wlsvp1}
\Ai0^\star = \argmin_{\Ai0 \in \spR^r} \| \tsX - S_{\tau}^\star(\Ai0) \|_{\bfW},
\end{equation}
Thus, for the numerical solution of the equation \eqref{eq:wls}, it is sufficient to consider iterations for the nonlinear part of the parameters.
This is the VP approach used in \cite{Usevich2012, Usevich2014}.

Let us denote $\bfJ_{S_{\tau}^\star}(\Ai0)$ the Jacobian matrix of $S_{\tau}^\star(\Ai0)$.
Then the iterations of the Gauss-Newton method for solving the problem \eqref{eq:wlsvp1} have the form
\begin{equation}
\label{eq:gauss_simple}
\Ai0^{(k+1)} = \Ai0^{(k)} + \gamma \winverse{\bfJ_{S_{\tau}^\star}(\Ai0^{(k)})}{\bfW} (\tsX - S_{\tau}^\star(\Ai0^{(k)})).
\end{equation}
The VPGN algorithm together with an explicit form of $\bfJ_{S_{\tau}^\star}(\Ai0^{(k)})$ is presented in Section~\ref{sec:VPGN_alg} (Algorithm~\ref{alg:gauss_newton_vp}).

\subsection{Modified Gauss-Newton method for \eqref{eq:wls} (MGN)}

In this section, we propose a new iterative method for the problem \eqref{eq:wls}, which is a modified Gauss-Newton method.

Let us return to the problem with the full set of parameters $(\Si0,\Ai0)$ and apply the approach that is described in Remark~\ref{rem:replacement},
 with $\widetilde S(\bfp) = S_{\tau}^\star(\Ai0)$. We can do it, since $S_{\tau}^\star(\Ai0) = \Proj_{\calZ(\fullop(\Ai0)), \bfW}(\tsX)$ and therefore \eqref{eq:minZ} is valid with $\calG(\Ai0) = \calZ(\fullop(\Ai0))$.
 Thus, we can consider $S_{\tau}^\star\big( \Ai0^{(k+1)}\big) \in \calD_r^\star$ as the result of the $(k+1)$-th iteration instead of
$S_{\tau}\big(\Si0^{(k+1)}, \Ai0^{(k+1)}\big) \in \overline \calD_r$.
It appears (see Section~\ref{sec:ZofA}) that then we can use more stable numerical calculations for the iteration implementation.
The proposed modification is similar to variable projections, since we can omit the part $\Si0$ of parameters.

Thus, we introduce the MGN iteration in the form
\begin{equation}
\label{eq:iterGNA}
\Ai0^{(k+1)}
= \Ai0^{(k)} + \gamma \left( \winverse{\bfJ_{S_{\tau}}(\Si0^{(k)}, \Ai0^{(k)})}{\bfW} (\tsX - S_{\tau}^\star(\Ai0^{(k)})\big) \right) _{\col{\{r+1, \ldots, 2r\}}},
\end{equation}
where $\Si0^{(k)}$ are taken as the corresponding boundary data from $S_\tau^\star(\Ai0^{(k)})$, i.e. $\Si0^{(k)} = \left( S_\tau^\star(\Ai0^{(k)}) \right)_{\calI(\tau)}$.
As well as in the variable projection method with the iteration step \eqref{eq:gauss_simple}, $S_\tau^\star(\Ai0^{(k+1)}) \in \calD_r^\star$ for each $k$.
\newcommand{\bfFs}{\bfF_{S, k}}
\newcommand{\bfFa}{\bfF_{\bfa, k}}

\begin{theorem}\label{th:equivalency}
Let $\bfW^{1/2} \bfJ_{S_{\tau}}(\Si0^{(k)}, \Ai0^{(k)})$ have full rank.
Denote $\bfS = \calT_{r+1} \left(\Proj_{\calZ(\fullop(\bfa^{(k)})), \bfW}(\tsX) \right)$,
$\bfM = - \left(\bfS_{\row{\calK({\i0})}}\right)^\rmT$.
Then the iteration step \eqref{eq:iterGNA} is equivalent to
\begin{equation}
\label{eq:iterGNfinal}
	\Ai0^{(k+1)}
	= \Ai0^{(k)} + \gamma \winverse{(\bfI_N - \Proj_{\calZ(\fullop(\Ai0^{(k)})), \bfW})\widehat \bfF_{\bfa}}{\bfW} (\tsX - \Proj_{\calZ(\fullop(\Ai0^{(k)})), \bfW}(\tsX)),
\end{equation}
where $\widehat \bfF_{\bfa} \in \spC^{N \times 2r}$ is an arbitrary matrix satisfying $\bfQ^\rmT(\fullop(\Ai0^{(k)})) \widehat \bfF_{\bfa} = \bfM$.
\end{theorem}
\begin{proof}
See the proof in Section~\ref{sec:th:equivalency}.
\end{proof}

\begin{remark}
In the case when $S_{\tau}(\Si0^{(k)}, \Ai0^{(k)}) \in \calD_r$, the Jacobian matrix $\bfJ_{S_{\tau}}(\Si0^{(k)}, \Ai0^{(k)})$ has full rank, according to Theorem~\ref{th:tangent}. This is sufficient for validity of the condition of Theorem~\ref{th:equivalency} if $\bfW$ has full rank. In the case of a rank-deficient matrix $\bfW$, the condition of Theorem~\ref{th:equivalency} is discussed in Remark~\ref{rem:degenerate} with $\bfF = \bfJ_{S_{\tau}}(\Si0^{(k)}, \Ai0^{(k)})$.
\end{remark}

Thus, we have constructed the version \eqref{eq:iterGNfinal} of the iteration step \eqref{eq:iterGNA} in such a way to reduce its complexity to the computational costs of computing the projections to $\calZ(\fullop(\Ai0)) = \calZ(\bfa)$ and calculating the matrices $\widehat \bfF_{\bfa}$ for different $\bfa$.
A numerically robust algorithm for calculating the iteration step \eqref{eq:iterGNfinal} is given in Section~\ref{sec:ZofA}.
The whole algorithm of the proposed MGN method is described in Algorithm~\ref{alg:gauss_newton_our}.

\section{Calculation of $\bfZ(\bfa)$ and $\widehat \bfF_{\bfa}$}
\label{sec:ZofA}
\label{fouriermethod}
For implementing the iteration step \eqref{eq:iterGNfinal} of the proposed optimization algorithm MGN, we need effective algorithms for calculating an
orthonormal basis $\bfZ(\bfa)$ of $\calZ(\bfa)$ together with calculating a matrix $\widehat \bfF_{\bfa}$ from \eqref{eq:iterGNfinal}.
 In this section, we consider the construction of
such orthonormal bases that allow one to calculate the projections in \eqref{eq:iterGNfinal} with improved precision.
Note that the constructed algorithms can also be used to improve the numerical stability of the iteration step \eqref{eq:gauss_simple} of the VPGN method.

\subsection{Circulant matrices and construction of $\bfZ(\bfa)$ and $\widehat \bfF_{\bfa}$}
{Let us start with the construction of $\bfZ(\bfa)$.}
Despite the series are real-valued, we construct a complex-valued basis of the complexification of $\calZ(\bfa)$, since this does not affect the result of the projection $\Proj_{\calZ(\bfa), \bfW} \bfv$ for any real
vector $\bfv$ and real matrix $\bfW$. Thus, we want to find a matrix $\bfZ(\bfa) = \bfZ \in \spC^{N \times r}$ of full rank to satisfy
$
\bfQ^\rmT(\bfa) \bfZ = \bm{0}_{(N-r) \times r}.
$

The matrix $\bfQ^\rmT(\bfa)$ is a partial circulant. Let us extend $\bfQ^\rmT(\bfa)$ to the circulant matrix $\bfC(\bfa)$ of $\bfa \in \spR^{r+1}$:
\begin{equation}\label{eq:circulant}
\bfC(\bfa) = \begin{pmatrix}
a_1 & a_2 & \dots & \dots & a_{r+1} & 0 & \dots & 0 \\
0 & a_1 & a_2 & \dots &  \dots & a_{r+1} & \ddots & \vdots \\
\vdots & \ddots  & \ddots & \ddots & \ddots & \ddots & \ddots & 0 \\
0 & \dots & 0 & a_1 & a_2 & \ddots & \ddots & a_{r+1} \\
a_{r+1} & \ddots & \ddots & \ddots & \ddots & \ddots & \ddots & \vdots \\
\vdots & \ddots & \ddots & \ddots & \ddots & \ddots & \ddots & \vdots \\
a_3& \dots & a_{r+1} & 0 & \dots & 0 & a_1 & a_2 \\
a_2& \dots & \dots & a_{r+1} & 0 & \dots  & 0 & a_1
\end{pmatrix}.
\end{equation}
Then $\bfv \in \calZ(\bfa)$ if and only if  $\bfC (\bfa) \bfv \in \span(\bfe_{N-r+1}, \ldots, \bfe_N)$, $\bfe_i\in \spR^{r+1}$.
If $\bfC(\bfa)$ has full rank, then we can find the basis vectors $\bfv_k$ solving the systems of linear equations
\be
\label{eq:lineqMGN}
    \bfC (\bfa) \bfv_k = \bfe_{N-k+1},\ k=1,\ldots,r,
\ee
with the computational cost of the order $O(r N \log N)$, since the calculations can be performed with the help of the discrete Fourier transform \cite{Davis2012} by fast Fourier transform (FFT), and then applying orthonormalization to the columns of $\bfV_r = \left[ \bfv_1 : \ldots :\bfv_r \right]$.

Let us apply the same approach to calculation of $\widehat \bfF_{\bfa}$ used in \eqref{eq:iterGNfinal}. According to Theorem~\ref{th:equivalency}, it is sufficient to find an arbitrary matrix such that $\bfQ^\rmT(\bfa) \widehat \bfF_{\bfa} = \bfM$, where $\bfM \in \spR^{(N-r) \times r}$ is defined in Theorem~\ref{th:equivalency}. Therefore, it is sufficient to solve the following systems
of linear equations:
\be
\label{eq:lineqMGN2}
    \bfC (\bfa) \widehat \bfF_{\bfa} = \left(\begin{matrix}\bfM\\\bfzero_{r \times r}\end{matrix}\right).
\ee

Denote $\calF_N$ and $\calF^{-1}_N$ the Fourier transform and the inverse Fourier transform for series of length $N$, respectively. That is,
for $\bfx = (x_0, \ldots, $ $x_{N-1})^\rmT \in \spC^N$ we have $\calF_N(\bfx) = \bfy = (y_0, \ldots, y_{N-1})^\rmT \in \spC^N$, where
$y_k = \frac{1}{\sqrt{N}} \sum_{j = 0}^{N-1} x_j \exp\big(-\frac{\unit 2 \pi k j}{N}\big)$.
Define $\calF_N(\bfX) = [\calF_N(\bfx_1): \ldots: \calF_N(\bfx_r)]$, where $\bfX=[\bfx_1 : \ldots : \bfx_r]$; the same for $\calF_N^{-1}(\bfY)$.

 Let
 \be
 \label{eq:pol_z}
 g_{\bfa}(z) = \sum_{k=0}^{r} a_{k+1}z^k
 \ee
  be the complex polynomial with coefficients $\bfa = (a_1,\ldots,a_{r+1})^\mathrm{T}$; we do not assume that the leading coefficient is non-zero.

The following lemma is a direct application of the theorem about the solution of a linear system of equations given by a circulant matrix \cite{Davis2012}.

\begin{lemma}
	\label{lemma:ev_circulant_pre}
1. Denote $\bfV_r = \calF_N^{-1}(\bfA_g^{-1} \bfR_r)$, where the matrices $\bfR_r = \calF_N([ \bfe_{N-r}:\ldots:\bfe_{N}])$ and $\bfA_g = \diag((g_\bfa(\omega_0), \ldots, g_\bfa(\omega_{N-1}))^\rmT)$ for $\omega_j = \exp\big(\frac{\unit 2 \pi  j}{N}\big)$.
Then $\bfQ^\rmT(\bfa) \bfV_r = \bm{0}_{(N-r) \times r}$, that is, $\colspace(\bfV_r) = \calZ(\bfa)$.
Herewith, the diagonal of the matrix $\bfA_g$ consists of the eigenvalues of the circulant matrix $\bfC(\bfa)$.\\ 2. Define $\widehat \bfF_\bfa = \calF_N^{-1}(\bfA_g^{-1} \widehat \bfR_r)$, where $\widehat \bfR_r = \calF_N \left( \left(\begin{matrix}\bfM\\\bfzero_{r \times r}\end{matrix}\right) \right)$. Then $\bfQ^\rmT(\bfa) \widehat \bfF_\bfa = \bfM$, i.e. $\widehat \bfF_\bfa$ satisfies the conditions of Theorem \ref{th:equivalency}.
\end{lemma}

\begin{remark} \label{rem:fourier_orthogon}
1. Let $\bfZ = \mathrm{orthonorm}(\bfV_r)$ be a matrix consisting of orthonormalized columns of the matrix $\bfV_r = \calF_N^{-1}(\bfA_g^{-1} \bfR_r)$ given in Lemma \ref{lemma:ev_circulant_pre}. Then $\bfZ$ is a matrix whose columns form an orthonormal basis of $\calZ(\bfa)$. Indeed, since $\bfQ^\rmT(\bfa)\bfV_r = \bfzero_{(N-r)\times r}$, we have $\bfQ^\rmT(\bfa)\bfZ = \bfzero_{(N-r)\times r}$.\\
2. Since $\calF_N^{-1}$ is a transformation which keeps orthonormality, the columns of the matrix calculated as $\bfZ = \calF_N^{-1}\left(\mathrm{orthonorm}(\bfA_g^{-1} \bfR_r)\right)$ also form an orthogonal basis of $\calZ(\bfa)$.

\end{remark}

\subsection{Shifting to improve conditioning}
Unfortunately, the circulant matrix $\bfC(\bfa)$ can be rank-deficient; e.g. in the case of the linear series $s_n=c_1 n + c_2$, which
is governed by the GLRR($\bfa$) with $\bfa = (1, -2, 1)^\rmT$. Therefore, instead of solving the linear systems \eqref{eq:lineqMGN} and \eqref{eq:lineqMGN2}, we consider similar systems with $\bfC(\widetilde{a})$, changing $a$ to $\widetilde{a}$ and then explain how use them to obtain the solutions of \eqref{eq:lineqMGN} and \eqref{eq:lineqMGN2}.

Lemma~\ref{lemma:ev_circulant_pre} shows that the eigenvalues of $\bfC(\bfa)$ coincide with the values of the polynomial $g_\bfa(z)$ in nodes of the equidistant grid $\calW = \left\{\exp\big(\frac{\unit 2 \pi j}{N}\big), \; j = 0, \ldots, N-1\right\}$ on the complex unit circle $\spT= \{z \in \spC : |z| = 1 \}$. Therefore, the nondegeneracy of $\bfC(\bfa)$ is equivalent to that there are no roots of the polynomial $g_\bfa(z)$ in $\calW$. The following lemma helps to avoid the problem with zero eigenvalues. Let us define the unitary matrix
\begin{equation}
\label{eq:eqi_grid}
\bfT_M(\alpha) = \diag\left((1, e^{\unit \alpha}, \ldots, e^{\unit (M-1) \alpha})^\rmT\right),
\end{equation}
where $\alpha$ is a real number, $M$ is a natural number.

\begin{lemma}\label{lemma:fourier_1}
	For any real $\alpha$, the following is true: $\bfQ^\rmT(\bfa) \bfx = \bfy$ is attained for some $\bfx \in \spC^N$, $\bfy \in \spC^{N-r}$ if and only if  $\bfQ^\rmT(\tilde \bfa) \left(\bfT_{N}(\alpha)\right) \bfx = \left(\bfT_{N-r}(\alpha)\right) \bfy$, where $\tilde \bfa = \tilde \bfa(\alpha) = \left(\bfT_{r+1}(-\alpha)\right) \bfa$. In addition, the eigenvalues of $\bfC(\tilde \bfa)$ are equal to $g_{\tilde \bfa}(\omega_j) = g_\bfa(\omega_j^{(\alpha)})$, where $\omega_j^{(\alpha)} = \omega_j e^{-\unit \alpha}$.
\end{lemma}
\begin{proof}
	The lemma directly follows from the definitions of the operator $\bfQ(\bfa)$ \eqref{op:Q} and the circulant matrix $\bfC(\bfa)$ \eqref{eq:circulant}.
\end{proof}

The equality $g_{\tilde \bfa}(\omega_j) = g_\bfa(\omega_j^{(\alpha)})$ means that the eigenvalues of $\bfC(\tilde \bfa(\alpha))$ coincide with
the values of the polynomial $g_\bfa(\tilde \omega)$ in $\tilde \omega \in \calW(\alpha) = \left\{\omega_j^{(\alpha)}, \; j = 0, \ldots, N-1\right\}$, where $\omega_j^{(\alpha)} = \exp \left(\unit \left(\frac{2 \pi j}{N} - \alpha \right) \right)$, $\calW(\alpha)$ is the $\alpha$-rotated equidistant grid on $\spT$
(it is sufficient to consider $-\pi/N < \alpha \le \pi/N$, since $\alpha$ and $\alpha +2 \pi/N$ yield the same rotated grid).
Therefore,  $\bfC(\tilde \bfa(\alpha))$ can be made non-degenerate by choosing a suitable  $\alpha$.

\begin{remark} \label{rem:fourier_rotat}
Lemma~\ref{lemma:fourier_1} provides a way for the calculation of an orthonormal basis of $\calZ(\bfa)$ together with the matrix $\widehat \bfF_\bfa$ from \eqref{eq:iterGNfinal}. Let us take $\alpha \in \spR$ such that $\bfC(\tilde \bfa)$ is non-degenerate for $\tilde \bfa = \tilde \bfa(\alpha)$.  Using Lemma \ref{lemma:ev_circulant_pre} and Remark~\ref{rem:fourier_orthogon}, we can obtain a matrix $\widetilde \bfZ$ formed from orthonormal basis vectors of $\calZ(\tilde \bfa)$, that is, $\bfQ^\rmT(\tilde \bfa) \widetilde \bfZ = \bm{0}_{(N-r) \times r}$ and $\colspace(\widetilde \bfZ) = \calZ(\widetilde \bfa)$. Then $\bfZ = \left(\bfT_N(-\alpha)\right) \widetilde \bfZ$ has orthonormal columns and $\bfQ^\rmT(\bfa)\bfZ = \bm{0}_{(N-r) \times r}$, that is, $\colspace(\bfZ) = \calZ(\bfa)$.
Similarly, for $\widetilde \bfF_a$ such that $\bfC (\bfa) \widetilde \bfF_{\bfa} = \left(\begin{matrix} \left(\bfT_{N-r}(\alpha)\right) \bfM\\\bfzero_{r \times r}\end{matrix}\right)$ and $\widehat \bfF_\bfa = \left(\bfT_N(-\alpha)\right) \widetilde \bfF_{\bfa}$, we have $\bfQ^\rmT(\bfa) \widehat \bfF_{\bfa} = \bfM$, that is, the conditions of Theorem \ref{th:equivalency} are satisfied.
\end{remark}

In the exact arithmetic, an arbitrary small non-zero value of the smallest eigenvalue of a matrix provides its non-degeneracy. However, in practice, the numerical stability and accuracy of matrix calculations depend on the condition numbers of matrices.
Therefore, the aim of the choice of a proper $\alpha$ is to do the condition number of $\bfC(\tilde \bfa(\alpha))$ as small as possible.
This minimization problem can be approximately reduced to the problem
of maximization of the smallest eigenvalue $|\lambda_\text{min}(\alpha)| = \min_{z \in \calW(\alpha)} | g_\bfa(z) |$ of $\bfC(\tilde \bfa(\alpha))$, since the maximal eigenvalue is not larger than $\max_{z \in \spT} |g_\bfa(z)|$.

\subsection{Algorithms}

By combining Lemmas \ref{lemma:ev_circulant_pre} and \ref{lemma:fourier_1} with Remarks \ref{rem:fourier_orthogon} and \ref{rem:fourier_rotat}, we obtain  Algorithm~\ref{alg:fourier_basis_A} for calculation of an orthonormal basis of $\calZ(\bfa)$.

\begin{algorithm}
	\caption{Calculation of a basis of $\calZ(\bfa)\subset \spC^N$}
	\label{alg:fourier_basis_A}
    \Input{$\bfa \in \spR^r$.}
	\begin{algorithmic}[1]
		\State Find $\alpha_0 = \argmax_{-\pi/N \le \alpha < \pi/N} \min_{z \in \calW(\alpha)} | g_\bfa(z) |$ by means of a 1D numerical optimization method.
		\State Calculate  the vector $\bfa_g = (a_{g, 0}, \ldots, a_{g, N-1})^\rmT$ consisting of the eigenvalues of $\bfC(\widetilde \bfa)$ by $a_{g, j} = g_\bfa\big(\exp(\unit (\frac{2 \pi j}{N} - \alpha_0)\big)$, $j = 0, \ldots, N-1$; $\bfA_g = \diag(\bfa_g)$.
		\State Calculate the matrices $\bfR_r = \calF_N([\bfe_{N-r+1}: \ldots: \bfe_N])$ and  $\bfL_r = \bfA_g^{-1} \bfR_r$.
		\State Find a matrix $\bfU_r \in \spC^{N \times r}$ consisting of orthonormalized columns of the matrix $\bfL_r$
        (e.g, $\bfU_r$ can be obtained by means of the QR decomposition of $\bfL_r$).
		\State Compute $\widetilde \bfZ = \calF_N^{-1}(\bfU_r)$.
		\State \Return $\bfZ = (\bfT_{N}(-\alpha_0)) \widetilde \bfZ \in \spC^{N\times r}$, whose columns form an orthonormal basis of $\calZ(\bfa)$, $\alpha_0$ and $\bfA_g$.
	\end{algorithmic}
\end{algorithm}

{Let us turn to calculating $\widehat \bfF_\bfa$ from \eqref{eq:iterGNfinal} in the same fashion.}
\begin{algorithm}
	\caption{Calculation of a matrix $\widehat \bfF_\bfa$ from \eqref{eq:iterGNfinal}}
	\label{alg:fourier_grad}
    \Input{$\bfa \in \spR^r$ and a series $\tsS\in \spR^N$ governed by the GLRR($\bfa$).}
	\begin{algorithmic}[1]
		\State{Compute $\alpha_0$, $\bfA_g$ using Algorithm \ref{alg:fourier_basis_A} (or Algorithm~\ref{alga:fourier_basis_A_comp} if an improved precision is necessary). }
        \State{Construct $\bfM = - (\bfS_{\row{\calK({\i0})}})^\rmT$, where $\bfS = \calT_{r+1} \left( \tsS \right)$.}
		\State{Calculate $\widetilde \bfM = \left(\begin{matrix} (\bfT_{N-r}(\alpha_0)) \bfM\\\bfzero_{r \times r}\end{matrix}\right)$.}
		\State{Calculate $\widehat \bfR_{r} = \calF_N(\widetilde \bfM)$ and  $\widetilde \bfF_\bfa = \calF_N^{-1}(\bfA_g^{-1} \widehat \bfR_{r})$.}
		\State\Return{$\widehat \bfF_\bfa = (\bfT_{N}(-\alpha_0)) \widetilde \bfF_\bfa  \in \spC^{N\times 2r}$.}
	\end{algorithmic}
\end{algorithm}

\begin{remark} \label{rem:A_g_C_a}
	Note that the use of the Fourier transform in Algorithm~\ref{alg:fourier_basis_A} allows us to avoid solving the system of linear equations with the matrix $\bfC(\tilde \bfa(\alpha))$. Instead, we invert the diagonal matrix $\bfA_g$, which has the same set of eigenvalues (and, therefore, the same condition number) as the matrix $\bfC(\tilde \bfa(\alpha))$.
\end{remark}

\subsection{Numerical properties of Algorithms~\ref{alg:fourier_basis_A} and \ref{alg:fourier_grad}}
Let us discuss the numerical behavior of the constructed algorithms.
The following theorem shows the order of the condition number of the circulant matrix $\bfC(\tilde \bfa(\alpha))$, where $\tilde \bfa(\alpha) = \left(\bfT_{r+1}(-\alpha)\right) \bfa$ is introduced in Lemma~\ref{lemma:fourier_1}, with respect to $\alpha$ in dependence on the series length $N$.
Conventionally, `big O' means an upper bound of the function order, while `big Theta' denotes the exact order.

\begin{theorem} \label{th:gamma}
Let $t$ be the maximal multiplicity of roots of the
polynomial $g_\bfa(z)$ on the unit circle $\spT$. Denote $\lambda_\text{min}(\alpha)$ the minimal eigenvalue of $\bfC (\tilde \bfa(\alpha))$ and $\lambda_\text{max}(\alpha)$ the maximal eigenvalue.
Then
	\begin{enumerate}
		\item for any real sequence $\alpha(N)$, $|\lambda_\text{min}(\alpha)| = O(N^{-t})$;
		\item for any real sequence $\alpha(N)$, $|\lambda_\text{max}(\alpha)| = \Theta(1)$;
		\item there exists such real sequence $\alpha(N)$ that $|\lambda_\text{min}(\alpha)| = \Theta (N^{-t})$.
	\end{enumerate}
\end{theorem}
\begin{proof}
See the proof in Section~\ref{sec:th:gamma}.
\end{proof}

Theorem~\ref{th:gamma} shows that the condition number of the matrix $\bfC(\tilde \bfa(\alpha))$ and $\bfA_g$ used in Algorithm \ref{alg:fourier_basis_A} can be considered as having the order $\Theta(N^{t})$.

The following remark is related to another possible improvement of the proposed algorithms.

\begin{remark}\label{rem:hornerscheme}
The calculation of a basis of $\calZ(\bfa)$ and $\widehat \bfF_\bfa$ can be an ill-conditioned problem if the polynomial $g_\bfa(z)$ has roots close to the unit circle $\spT$, see Theorem~\ref{th:gamma}.
Therefore, we propose to use the error-free arithmetics and the compensated Horner scheme \cite[Algorithm \code{CompHorner}]{Graillat2008} for Algorithms \ref{alg:fourier_basis_A} and \ref{alg:fourier_grad}; see Section~\ref{subsec:hornerscheme}.
\end{remark}

\section{Algorithms of Variable Projection Gauss-Newton  (VPGN) and Modified Gauss-Newton (MGN) methods}
\label{sec:MGNand VPGN}
\subsection{Calculation of weighted projection to subspace with a given basis}
\label{sec:Pi_weighted}
    The MGN iteration step \eqref{eq:iterGNfinal} uses the projections $\Proj_{\bfZ, \bfW} \bfx$ for a vector $\bfx \in \spC^N$, where the matrix $\bfZ$ belongs to $\spC^{N \times r}$, while the VPGN iteration step \eqref{eq:gauss_simple} uses the projections $\Proj_{\bfZ, \bfW} \bfx$ for real $\bfx$ and $\bfZ$, which is real or complex depending on the implementation details.
     We assume that if the matrix $\Sigminus$ is $(2p+1)$-diagonal and positive definite, then it is presented in the form of
  the Cholesky decomposition $\Sigminus = \bfC^\rmT \bfC$; here $\bfC$ is an upper triangular matrix with $p$ nonzero superdiagonals \cite[p. 180]{GoVa13}.
  If $\Sigminus^{-1}$ is $(2p+1)$-diagonal and positive definite, then we consider the representation $\Sigminus = \widehat \bfC^{-1} (\widehat \bfC^{-1})^\rmT$, where $\Sigminus^{-1} = \widehat \bfC^\rmT \widehat \bfC$ is the Cholesky decomposition of $\Sigminus^{-1}$; here $\widehat\bfC$ is an upper triangular matrix with $p$ nonzero superdiagonals.

\begin{remark}\label{rem:wlsinfourier}
	As we mentioned in the beginning of Section~\ref{sec:optim}, the calculation of pseudoinverses ($\inverse{(\bfC \bfZ)}$ or $\inverse{((\widehat \bfC^{-1})^\rmT \bfZ)}$ in our case) can be reduced to solving a linear weighted least-squares problem and therefore their computing can be performed with the help of either the QR factorization or the SVD of the matrix $\bfC \bfZ$ or $(\widehat \bfC^{-1})^\rmT \bfZ$ respectively.
\end{remark}

\begin{algorithm}
	\caption{Calculation of $\winverse{\bfZ}{\bfW}$ and $\Proj_{\bfZ, \bfW} \bfx$ with the use of $\Sigminus = \bfC^\rmT \bfC$ or $\Sigminus^{-1} = \widehat \bfC^\rmT \widehat \bfC$}
	\label{alg:proj_calc}
    \Input{$\bfZ \in \spC^{N\times r}$, $\bfW \in \spR^{N\times N}$ and $\bfx \in \spC^N$.}
	\begin{algorithmic}[1]
		\If{$\Sigminus$ is $(2p+1)$-diagonal}
		\State{Compute the vector $\bfC \bfx$ and the matrix $\bfC \bfZ$.}
		\State{Calculate $\bfq = \inverse{(\bfC \bfZ)} (\bfC \bfx)$, see Remark \ref{rem:wlsinfourier}.}
		\EndIf
		\If{$\Sigminus^{-1}$ is $(2p+1)$-diagonal}
		\State{Compute the vector $(\widehat \bfC^{-1})^\rmT \bfx$ and the matrix $(\widehat \bfC^{-1})^\rmT \bfZ$.}
		\State{Calculate $\bfq = \inverse{((\widehat \bfC^{-1})^\rmT \bfZ)} ((\widehat \bfC^{-1})^\rmT \bfx)$, see Remark \ref{rem:wlsinfourier}.}
		\EndIf
		\State\Return{$\winverse{\bfZ}{\bfW} = \bfq \in \spR^{r\times N}$ and $\Proj_{\bfZ, \bfW} \bfx = \bfZ \bfq  \in \spR^{N}$.}
	\end{algorithmic}
\end{algorithm}

\begin{remark}\label{rem:wlsseminorm}
\label{rem:wlsseminorm_ts}
Algorithm \ref{alg:proj_calc} can be applied to the case of positive semidefinite weight matrices;
although, the general case of the Cholesky factorization of a degenerate matrix $\bfW$ is complicated, see \cite[p. 201]{higham2002accuracy}).
	However, there is a particular case of degenerate weight matrices, which corresponds to a time series with missing values. As we mentioned in Section~\ref{sec:intro}, in the presence of missing values, the weight matrix $\bfW$ has zero columns and rows corresponding to missing entries, which can be easily processed. Denote $\bfu  \in \spR^N$ the vector with units at the places of observations and zeros at the places of missing values, $\bfU = \diag(\bfu)$. Then the matrix $\bfW$ can be expressed as $\bfW = \bfU^\rmT \bfW_0 \bfU$. Suppose that $\bfW_0$ is positive definite. Consider the Cholesky decomposition $\bfW_0 = \bfC_0^\rmT \bfC_0$, where $\bfC_0$ is upper triangle, and set $\bfC = \bfC_0 \bfU$.
Then $\bfW = \bfC^\rmT \bfC$.
Note that if $\bfC_0$ is upper triangular with $p$ nonzero superdiagonals, then $\bfC$ is also upper triangular and has $p$ nonzero superdiagonals.
\end{remark}

\subsection{Calculation of $\Proj_{\calZ(\bfa), \bfW} \bfx$}

Calculating the projection $\Proj_{\calZ(\bfa), \bfW} \bfx$ can be performed either directly onto $\calZ(\bfa)$ with the use of its specific features (see Section~\ref{sec:ZofA}) or by means of constructing the projection $\Proj_{\calQ(\bfa), \bfW}$ onto the orthogonal compliment $\calQ(\bfa)$ (as suggested in \cite{Usevich2014}) and then subtracting from the identity matrix: $\bfI_N - \Proj_{\calQ(\bfa), \bfW}$.

Let us start with the algorithm proposed in \cite{Usevich2014}.
The calculation of $\Proj_{\calZ(\bfa), \Sigminus} \bfx$  in \cite{Usevich2014} is performed by means of the relation
\begin{equation} \label{eq:spZa_kostya}
	\Proj_{\calZ(\bfa), \bfW} \bfx = \left( \bfI_N - \bfW^{-1} \bfQ(\bfa) \bm\Gamma^{-1}(\bfa) \bfQ^\rmT(\bfa) \right) \bfx,
	\end{equation}
	where $\bm\Gamma(\bfa) = \bfQ^\rmT(\bfa) \Sigminus^{-1} \bfQ(\bfa)\in \spR^{(N-r)\times (N-r)}$ (see Lemma~\ref{th:varproj}).
The calculation of $\Proj_{\calZ(\bfa), \bfW}$ by \eqref{eq:spZa_kostya} needs computing the matrix  $\bm \Gamma^{-1}(\bfa)$.
Below we write down Algorithm~\ref{alg:solution_calc_vp}, which was used in the paper \cite{Usevich2014}, with a fast computation of $\bm \Gamma(\bfa)$ and its inverse (see Algorithm~\ref{alg:gamma_inverse}). Algorithm~\ref{alg:gamma_inverse} uses the matrix $\widehat \bfC$, which is defined as in the beginning of Section~\ref{sec:Pi_weighted}, i.e. $\Sigminus^{-1} = \widehat \bfC^\rmT \widehat \bfC$ is the Cholesky decomposition of $\Sigminus^{-1}$. Note that Algorithm~\ref{alg:solution_calc_vp} is applied to the case of a positive definite $\bfW$ only; the case of degenerate weight martices is not considered here since it requires a completely different algorithm, see \cite{Markovsky2013missing}.

\begin{algorithm}
	\caption{Calculation of $\bm\Gamma^{-1}(\bfa) \bfv$ by the method from \cite{Usevich2014}}
	\label{alg:gamma_inverse}
    \Input{$\bfa \in \spR^{r}$, $\bfW \in \spR^{N\times N}$, $\bfW^{-1}$ is $(2p+1)$-diagonal ($p\leq N$), $\bfv \in \spR^{N-r}$}
	\begin{algorithmic}[1]
		\State{Calculate the matrix $\widehat \bfC \bfQ(\bfa)$, which has $m+1$ non-zero diagonals, where $m=\min(p+r, N-r)$.}
		\State{Calculate $(2m+1)$-diagonal matrix $\bm\Gamma(\bfa) = (\widehat \bfC \bfQ(\bfa))^\rmT(\widehat \bfC \bfQ(\bfa))$.}
		\State{Calculate the Cholesky decomposition $\bm\Gamma(\bfa) = (\bm\Gamma_\mathrm{c})^\rmT \bm\Gamma_\mathrm{c}$, where $\bm\Gamma_\mathrm{c}$ is $(m+1)$-diagonal.}
		\State\Return $\bm\Gamma^{-1}(\bfa)\bfv = \bm\Gamma_\mathrm{c}^{-1}\left( (\bm\Gamma_\mathrm{c}^\rmT)^{-1}\bfv \right)$
	\end{algorithmic}
\end{algorithm}

Algorithm~\ref{alg:gamma_inverse} is used for calculating the projection $\Proj_{\calZ(\bfa), \bfW}$ in Algorithm~\ref{alg:solution_calc_vp} in the way similar to that in \cite{Usevich2014}.

\begin{algorithm}
	\caption{Calculation of $\Proj_{\calZ(\bfa), \bfW} \tsX$ by the method from \cite{Usevich2014} using \eqref{eq:spZa_kostya}}
	\label{alg:solution_calc_vp}
	\Input{$\bfa \in \spR^{r}$, $\bfW \in \spR^{N\times N}$.}
	\begin{algorithmic}[1]
		\State{Compute $\bfv = \bfQ^\rmT(\bfa) \tsX \in \spR^{N-r}$}
		\State{Compute $\bfy = \bm\Gamma^{-1}(\bfa) \bfv \in \spR^{N-r}$ using Algorithm~\ref{alg:gamma_inverse}}
		\State\Return{$\Proj_{\calZ(\bfa), \bfW} \tsX = \tsX - \bfW^{-1}\bfQ^\rmT(\bfa)\bfy$}
	\end{algorithmic}
\end{algorithm}

The theory described in Section~\ref{sec:ZofA} allows us to improve Algorithm~\ref{alg:solution_calc_vp}. The proposed method is described in Algorithm~\ref{alg:solution_calc_our}.

\begin{algorithm}
	\caption{Calculation of $\Proj_{\calZ(\bfa), \bfW} \tsX$ with the use of special properties of $\calZ(\bfa)$}
	\label{alg:solution_calc_our}
    \Input{$\bfa \in \spR^{r}$, $\bfW \in \spR^{N\times N}$.}
	\begin{algorithmic}[1]
		\State{Compute the matrix $\bfZ(\bfa)$ consisting of basis vectors of $\calZ(\bfa)$ by Algorithm \ref{alg:fourier_basis_A} (or Algorithm~\ref{alga:fourier_basis_A_comp} if an improved precision is necessary).}
		\State\Return{Calculate $\Proj_{\bfZ(\bfa), \bfW} \tsX$ by means of Algorithm~\ref{alg:proj_calc}.}
	\end{algorithmic}
\end{algorithm}

The use of direct projecting to $\calZ(\bfa)$ in Algorithm~\ref{alg:solution_calc_our} gives the advantage over Algorithm~\ref{alg:solution_calc_vp}, since Algorithm~\ref{alg:solution_calc_our} can be applied to the case of a rank-deficient matrix $\bfW$ (see Remark~\ref{rem:wlsseminorm_ts}).

\subsection{The VPGN algorithm}
\label{sec:VPGN_alg}
Algorithm~\ref{alg:gauss_newton_vp} implements the iterations \eqref{eq:gauss_simple}, which were obtained in \cite{Usevich2014} by the variable projection approach (see Section~\ref{sec:VPGN}). As will be discussed in Section~\ref{sec:comp_cost}, an effective implementation of this algorithm (the fast calculation of $\Proj_{\bfZ(\bfa), \bfW} \bfx$ within the algorithm) is available if $\Sigminus^{-1}$ is $(2p+1)$-diagonal.
Recall notation: $\fullop$ and $S_\tau^\star(\Ai0) = \Proj_{\calZ(
\bfa), \bfW} \tsX$, where $\bfa = \fullop(\Ai0)$, are introduced in Sections~\ref{sec:param} and \ref{sec:VPGN} respectively.

We present a new form for the Jacobian $\bfJ_{\tsS_\tau^\star}$ (see Lemma~\ref{th:varproj}), which is more suitable for implementation.
The columns of $\bfJ_{\tsS_\tau^\star}(\bfa)$ has the form
\begin{equation} 	\label{eq:vpformula}
	\left(\bfJ_{\tsS_\tau^\star}(\Ai0) \right)_{\col{i}} = -\bfW^{-1} \bfQ(\bfa) \bm\Gamma^{-1}(\bfa) \bfQ^\rmT(\bfe_j) \Proj_{\calZ(\bfa), \bfW} \tsX
	-\Proj_{\calZ(\bfa), \bfW} \bfW^{-1} \bfQ(\bfe_j) \bm\Gamma^{-1}(\bfa) \bfQ^\rmT(\bfa) \tsX,
\end{equation}
	where $\bfa = \fullop(\Ai0)$ and $j = (\calK({\i0}))_i$ is the $i$-th element of $\calK({\i0})$.

\begin{algorithm}
	\caption{Variable Projection Gauss-Newton method (VPGN)}
	\label{alg:gauss_newton_vp}
    \Input{$\tsX \in \spR^N$, $\bfa_0\in \spR^{r+1}$, a stopping criterion STOP.}
	\begin{algorithmic}[1]
		\State{Set $k = 0$,  $\bfb^{(0)} = \bfa_0$.}
		\Repeat{}
		\State{Choose $\i0$ such that $b_\i0^{(k)}\neq 0$; for example, find $\i0 = \argmax_{i} |b_i^{(k)}|$. Calculate $\bfa^{(k)} = c\bfb^{(k)}$, where $c$ is such that $c b_\tau^{(k)} = -1$, and take $\Ai0^{(k)} = \fullop^{-1}(\bfa^{(k)})$.}
		\State{Calculate $\tsS_k  = S_\tau^\star(\Ai0^{(k)})$  using Algorithm~\ref{alg:solution_calc_vp} or Algorithm~\ref{alg:solution_calc_our} to compute $\Proj_{\calZ(\bfa^{(k)}) , \bfW}$.}
		\State{Calculate $\bfJ_{S_\tau^\star}(\Ai0^{(k)})$ by \eqref{eq:vpformula} applying Algorithm~\ref{alg:gamma_inverse} to compute $\bm\Gamma^{-1}(\bfa^{(k)}) \bfv$ for $\bfv = \bfQ^\rmT(\bfe_j) \Proj_{\calZ(\bfa^{(k)}), \bfW} \tsX$ and for  $\bfv = \bfQ^\rmT(\bfa^{(k)}) \tsX$.}
		\State{Calculate
			$\Delta_k = \winverse{\bfJ_{S_\tau^\star}(\Ai0^{(k)})}{\bfW} (\tsX - \tsS_k)$} by applying Algorithm~\ref{alg:proj_calc} to computing the pseudoinverse.
		\State{Perform a step of size $\gamma_k$ for the line search in the descent direction given by $\Delta_k$ using Algorithm~\ref{alg:solution_calc_vp} or Algorithm~\ref{alg:solution_calc_our} to calculate $S_\tau^\star$. For example, find $0\le\gamma_k\le 1$ such that
          $\|\tsX - S_{\tau}^\star(\Ai0^{(k)} + \gamma \Delta_k) \|_{\bfW} \le \|\tsX - S_\tau^\star(\Ai0^{(k)}) \|_{\bfW}$
          by the backtracking method \cite[Section 3.1]{nocedal2006numerical}.}
		\State{Set $\Ai0^{(k+1)} = \Ai0^{(k)} + \gamma_k \Delta_k$, $\bfb^{(k+1)} = \fullop(\Ai0^{(k+1)})$.}
		\State{Set $k = k+1$.}
		\Until{STOP}
		\State\Return{$\widetilde \tsS = S_{\i0}^\star(\Ai0^{(k)})$ as an estimate of the signal.}
	\end{algorithmic}
\end{algorithm}

Algorithm~\ref{alg:gauss_newton_vp} can be implemented in two versions, with the projection $S_\tau^\star(\Ai0) = \Proj_{\calZ(\fullop(\Ai0)), \bfW} \tsX$ calculated by either Algorithm~\ref{alg:solution_calc_vp} or Algorithm~\ref{alg:solution_calc_our}.
The former version of the algorithm was proposed in \cite{Usevich2014}, whereas the latter version is more numerically stable.
Even with the use of Algorithm~\ref{alg:solution_calc_our}, Algorithm~\ref{alg:gauss_newton_vp} is hardly
extended to the case of a degenerate $\bfW$, since it still includes the call of  Algorithm~\ref{alg:gamma_inverse}.

\subsection{The MGN algorithm}

Algorithm~\ref{alg:gauss_newton_our} implements the iterations \eqref{eq:iterGNfinal} of the Modified Gauss-Newton algorithm, which is proposed in this paper. This algorithm uses  Algorithm~\ref{alg:solution_calc_our} for calculating $S_\tau^\star(\Ai0) = \Proj_{\calZ(\fullop(\Ai0)), \bfW} \tsX$ and differs from Algorithm~\ref{alg:gauss_newton_vp} mainly by steps 5 and 6.

\begin{algorithm}
	\caption{Modified Gauss-Newton method (MGN)}
	\label{alg:gauss_newton_our}
    \Input{$\tsX \in \spR^N$, $\bfa_0\in \spR^{r+1}$, a stopping criterion STOP.}
	\begin{algorithmic}[1]
		\State{Set $k = 0$, $\bfb^{(0)} = \bfa_0$.}
		\Repeat{}
		\State{Choose $\i0$ such that $b_\i0^{(k)}\neq 0$; for example, find $\i0 = \argmax_{i} |b_i^{(k)}|$. Calculate $\bfa^{(k)} = c\bfb^{(k)}$, where $c$ is such that $c b_\tau^{(k)} = -1$, and take $\Ai0^{(k)} = \fullop^{-1}(\bfa^{(k)})$.}
		\State{Calculate $\tsS_k  = S_\tau^\star(\Ai0^{(k)})$  using Algorithm~\ref{alg:solution_calc_our} to compute $\Proj_{\calZ(\bfa^{(k)}) , \bfW}$.}
\State{Calculate $\widehat \bfF_{\bfa^{(k)}}$ by Algorithm \ref{alg:fourier_grad} with $\bfa = \bfa^{(k)}$
and $\tsS=\tsS_k$.}
\State{Calculate
$\Delta_k = \winverse{\left(\bfI_N - \Proj_{\calZ(\bfa^{(k)}), \bfW}\right)\widehat \bfF_{\bfa^{(k)}}}{\bfW} (\tsX - \tsS_k)$  by applying Algorithm~\ref{alg:proj_calc} to computing the pseudoinverse.}
		\State{Perform a step of size $\gamma_k$ for the line search  in the descent direction given by $\Delta_k$ using Algorithm~\ref{alg:solution_calc_our} to calculate $S_\tau^\star$. For example, find $0\le\gamma_k\le 1$ such that
          $\|\tsX - S_{\tau}^\star(\Ai0^{(k)} + \gamma \Delta_k) \|_{\bfW} \le \|\tsX - S_\tau^\star(\Ai0^{(k)}) \|_{\bfW}$
          by the backtracking method \cite[Section 3.1]{nocedal2006numerical}.}
		\State{Set $\Ai0^{(k+1)} = \Ai0^{(k)} + \gamma_k \Delta_k$, $\bfb^{(k+1)} = \fullop(\Ai0^{(k+1)})$.}
\State{Set $k = k+1$.}
\Until{STOP}
		\State\Return{$\widetilde \tsS = S_{\i0}^\star(\Ai0^{(k)})$ as an estimate of the signal.}
	\end{algorithmic}
\end{algorithm}

\newpage
\section{Comparison of optimization algorithms}
\label{sec:comparison}
Let us compare Algorithm~\ref{alg:gauss_newton_vp} of the VPGN method and Algorithm~\ref{alg:gauss_newton_our} of the proposed MGN method from the computational viewpoint.

\subsection{Design of comparison} \label{sec:allalgorithms}
We consider four versions of the algorithms:
\begin{enumerate}
	\item the method VPGN (Algorithm~\ref{alg:gauss_newton_vp}), where the projections $\Proj_{\calZ(\bfa), \bfW}$ are calculated by Algorithm~\ref{alg:solution_calc_vp} as in \cite{Usevich2014};
	\item the method S-VPGN (Algorithm~\ref{alg:gauss_newton_vp}), where the projections $\Proj_{\calZ(\bfa), \bfW}$ are calculated by Algorithm~\ref{alg:solution_calc_our}; the Compensated Horner scheme is used;
	\item the proposed method MGN (Algorithm~\ref{alg:gauss_newton_our}), where the projections $\Proj_{\calZ(\bfa), \bfW}$ are calculated by Algorithm~\ref{alg:solution_calc_our};  the Compensated Horner scheme is not used.
	\item the proposed method S-MGN (Algorithm~\ref{alg:gauss_newton_our}), where the projections $\Proj_{\calZ(\bfa), \bfW}$ are calculated by Algorithm~\ref{alg:solution_calc_our}; the Compensated Horner scheme is used.
\end{enumerate}

Figure~\ref{fig:computational_scheme0} shows the schemes of calls of the algorithms used for the implementation of VPGN
and MGN, whereas Figure~\ref{fig:computational_scheme_stable} shows how these schemes are changed if we consider
the more stable versions S-VPGN and S-MGN.
The red color corresponds to VPGN, the blue color corresponds to MGN and the magenta color serves for the algorithms that are used by both methods.

These algorithms were implemented with the help of \code{R} and \code{C++}; the source code can be found in
\cite{Zvonarev2019}. 
In addition to the MGN method, the VPGN method, which is applied to the case of a common (not necessarily diagonal) weight matrix $\bfW$, was implemented in \cite{Zvonarev2019}. This implementation extends that from \cite{Markovsky.Usevich2014}, which is suitable for diagonal weight matrices only, and has the same order of computational cost.

We compare the MGN and VPGN algorithms theoretically and numerically, while S-MGN and S-VPGN are compared only numerically.

\begin{figure}[!hbt]
	\centering
{\small
    \begin{tikzpicture}
\tikz {
	\node (7) at (1.5,3) [rounded corners, fill = red, text=white] {Algorithm~7 (VPGN)
	};
	\node (4) at (4.5,4) [rounded corners, fill = red, text=white] {Algorithm~4};
	\node (5) at (4.5,3) [rounded corners, fill = red, text=white] {Algorithm~5};
    \node (8) at (1.5,1) [rounded corners, fill = blue, text=white] {Algorithm~8 (MGN)};
	\node (6) at (4.5,1) [rounded corners, fill = blue, text=white] {Algorithm~6};
	\node (1) at (7.5,1) [rounded corners, fill = blue, text=white] {Algorithm~1};
	\node (2) at (4.5,0) [rounded corners, fill = blue, text=white] {Algorithm~2};
	\node (3) at (7.5,2) [rounded corners, fill = violet, text=white] {Algorithm~3};
	\draw (7) edge[->] (4) (7) edge[->] (5)
	(5) edge[->] (4) (8) edge[->] (6) (6) edge[->] (1)
	(8) edge[->] (2) (2) edge[->] (1)
	(7) edge[->] (3) (8) edge[->] (3) (6) edge[->] (3);
}
\end{tikzpicture}
}
	\caption{Computational scheme for VPGN and MGN}

	\label{fig:computational_scheme0}
\end{figure}

\begin{figure}[!hbt]
	\centering
{\small
    \begin{tikzpicture}
\tikz {
	\node (7) at (1.5,3) [rounded corners, fill = red, text=white] {Algorithm~7 (S-VPGN)};
	\node (4) at (4.5,4) [rounded corners, fill = red, text=white] {Algorithm~4};
    \node (8) at (1.5,1) [rounded corners, fill = blue, text=white] {Algorithm~8 (S-MGN)};
	\node (6) at (4.5,1) [rounded corners, fill = violet, text=white] {Algorithm~6};
	\node (1) at (7.5,1) [rounded corners, fill = violet, text=white] {
		Algorithm~1
	};
	\node (2) at (4.5,0) [rounded corners, fill = blue, text=white] {Algorithm~2 (stable)};
	\node (3) at (7.5,3) [rounded corners, fill = violet, text=white] {Algorithm~3};
	\draw (7) edge[->] (4)
	(8) edge[->] (6) (6) edge[->] (1)
	(8) edge[->] (2) (2) edge[->] (1)
	(7) edge[->] (3) (8) edge[->] (3) (6) edge[->] (3)
	(7) edge[->] (6);
}
\end{tikzpicture}
}
	\caption{Computational schemes for S-VPGN and S-MGN}
	\label{fig:computational_scheme_stable}
\end{figure}

\subsection{Theoretical comparison}
We start the comparison from comparing the algorithms by the computational costs. Then, we will compare the stability of the algorithms in the conditions, when the algorithms are comparable by the computational costs.
This depends on the structure of the weight matrix $\bfW$.
The special case of interest is the case when the weight matrix  $\Sigminus$ is $(2p+1)$-diagonal with a small $p$ (this is the case of autoregressive noise and therefore a natural assumption). Note that a special case when both $\Sigminus$
and $\Sigminus^{-1}$ are banded corresponds to the case of a diagonal matrix $\Sigminus$.

\subsubsection{Computational cost}
\label{sec:comp_cost}
Let us estimate computational costs in flops and study the asymptotic costs as $N \rightarrow \infty$.
The proposed algorithms can be divided into several standard operations with known computational costs. We will use the following asymptotic orders: FFT of a sequence of length $N$ takes $O(N \log N)$ flops \cite[Chapter 1.4.1]{GoVa13}, FFT of the unit vector of length $N$ takes $\Theta(N)$ flops; the Cholesky decomposition of a $(2p+1)$-diagonal matrix $\bfA\in\spR^{N\times N}$ takes $\Theta(N (p+1)^2)$ flops \cite[Chapter 4.3.5]{GoVa13};  solving the  system of linear equations $\bfA\bfx = \bfb$ with $\bfb \in \spR^{N}$ using the obtained decomposition takes additionally $\Theta(N (p+1))$ flops, whereas for $\bfA\bfX = \bfB$ with $\bfB \in \spR^{N\times r}$ the additional cost is $\Theta(N r (p + 1))$ flops; the QR decomposition of an $N\times r$ matrix of rank $r$ takes $\Theta(N r^2)$ flops \cite[Chapter 5.2]{GoVa13}; the pseudo-inversion has the same cost as the QR decomposition, see Remark~\ref{rem:wlsinfourier}; the cost of matrix multiplication is directly determined by their size and structure, in particular, the multiplication of a $(2p+1)$-diagonal $N\times N$ matrix by a vector takes $\Theta(N (p + 1))$ flops \cite[Chapter 1.2.5]{GoVa13}, where $p=0$ corresponds to the case of a diagonal matrix; the computation of a polynomial of order $r$ at $N$ given points takes $\Theta(Nr)$ flops.

\paragraph{The MGN method}

Although the implementations of Algorithm~\ref{alg:gauss_newton_our} differ for the case when $\Sigminus^{-1}$ is $(2p+1)$-diagonal
and the case when $\Sigminus$ is $(2p+1)$-diagonal, the asymptotic computational cost is the same.

Algorithm~\ref{alg:fourier_basis_A} includes computing the $N \times N$ diagonal matrix $\bfA_g$, where each diagonal value is obtained using the calculation of a polynomial of order $r$ (step 2); solving a system of linear equations given by a diagonal matrix (step 3); FFT of $r$ unit vectors (step 3); FFT of $r$ arbitrary vectors (step 5); the QR decomposition (step 4); the multiplication of a diagonal matrix by a vector $r$ times (step 6). The search of optimal rotations at step 1 of Algorithm~\ref{alg:fourier_basis_A} serves for increasing of the algorithm stability. Therefore we can take a fix number of iterations in this search. Since the computational cost of calculating the objective function is  $O(N r)$ flops, the cost of step 1 is also $O(N r)$ flops. Therefore, Algorithm~\ref{alg:fourier_basis_A} requires $O(r N \log N + N r^2)$ flops, or $O(N \log N)$ for a fixed $r$. Algorithm~\ref{alga:fourier_basis_A_comp}, which is used in the stable version of MGN instead of Algorithm~\ref{alg:fourier_basis_A}, additionally uses matrix multiplications of an $N \times r$ matrix by a $r \times r$ one and solving a diagonal system at step 3, which leads to the same asymptotic as Algorithm~\ref{alg:fourier_basis_A}.

  Algorithm~\ref{alg:fourier_grad} includes the first two steps of Algorithm~\ref{alg:fourier_basis_A};  constructing matrices with the use of the multiplication by a diagonal matrix (steps 2-3); FFT and solving a system of linear equations given by a diagonal matrix (step 4); the multiplication of a diagonal matrix by a vector $2r$ times (step 5); all of them give $O(r N \log N)$ flops in sum.  Calculating the projection by Algorithm~\ref{alg:proj_calc} includes the multiplication by a $p$-diagonal matrix and the QR decomposition for the pseudoinverse compution that leads to $\Theta(N r^2 + N rp)$ operations. Therefore, the asymptotic cost of Algorithm~\ref{alg:solution_calc_our} and finally of one iteration of Algorithm~\ref{alg:gauss_newton_our} is $O(N r^2 + Np^2 + r N \log N)$, or $O(Np^2 + N \log N)$ for a fixed rank $r$. This order includes $\Theta(N (p + 1)^2)$ flops needed for computing the Cholesky decompositions of either matrix $\Sigminus^{-1}$ or $\Sigminus$.

\paragraph{The VPGN method}

Let $\Sigminus^{-1}$ be $(2p+1)$-diagonal.
Algorithm~\ref{alg:gamma_inverse} includes computing the Cholesky factorization of a ($2m+1$)-diagonal matrix of order $(N-r) \times (N-r)$, where $m \le p + r$ (step 3); solving a system of linear equations using the obtained decomposition (step 4); the multiplications of matrices with $p+1$ and $r+1$ non-zero diagonals (step 1), $(m+1)$ and $(m+1)$ non-zero diagonals (step 2). This gives us the asymptotic cost $O(N r^2 + N p^2)$ flops. Algorithm~\ref{alg:solution_calc_vp} includes the multiplications by a $(r+1)$-diagonal matrix (step 1, step 3), a $(2m+1)$-diagonal matrix (step 2), and a $(2p+1)$-diagonal matrix (step 3) with the call of Algorithm~\ref{alg:gamma_inverse} at step 2. Therefore, the asymptotic cost is the same as for Algorithm~\ref{alg:gamma_inverse}.
Thus, the asymptotic cost of one iteration of Algorithm~\ref{alg:gauss_newton_vp} is also $O(N r^2 + Np^2)$, or $O(N + N p^2)$ for a fixed $r$ (we assume that the Cholesky decomposition of the matrix $\bm\Gamma(\bfa^{(k)})$ in Algorithm~\ref{alg:gamma_inverse} is performed once for one iteration).

For the case when $\Sigminus$ is $(2p+1)$-diagonal, $p > 0$, there is no implementation of Algorithm~\ref{alg:gamma_inverse} faster than with cubic (in $N$) asymptotic complexity, since $\bm \Gamma(\bfa)$ (see Section~\ref{sec:MUdetails}) is not a banded matrix.
Therefore, the complexity of Algorithm~\ref{alg:gauss_newton_vp} is $O(N^3)$.

\begin{remark}
\label{rem:compcost}
Thus, if the inverse $\Sigminus^{-1}$ of the weight matrix $\Sigminus$ is $(2p+1)$-diagonal, then the computational cost of the proposed MGN method is slightly larger in comparison with the VPGN method.
However, if the weight matrix  $\Sigminus$ is $(2p+1)$-diagonal and $p>0$ (this is the case of autoregressive noise and therefore a natural assumption),
then the computational cost of the MGN method is significantly smaller by order. In the case of a diagonal matrix $\bfW$, the costs of MGN and VPGN are $O(N \log N)$ and $O(N)$ respectively.
\end{remark}

\subsubsection{Stability}
\label{sec:comp_thstab}
Let us focus on the main ``stability bottlenecks'' of both methods, which consist of solving the systems of linear equations with matrices related to $\bfQ(\bfa)$. In fact, we say about inverting the matrices depending on $\bfa$, that is, on the coefficients of GLRR($\bfa$) governing the signal. For the MGN algorithm, it is the matrix $\bfA_g$ inverted in Algorithm~\ref{alg:fourier_basis_A}; for the VPGN algorithm, it is the matrix $\bm\Gamma(\bfa)$ whose inversion is constructed in Algorithm~\ref{alg:gamma_inverse}. Let us compare the  orders of the condition numbers of these matrices as the time-series length $N$ tends to infinity.

\paragraph{The MGN method}

Recall that the inversion of the matrix $\bfA_g$ in Algorithm~\ref{alg:fourier_basis_A} and Algorithm \ref{alg:fourier_grad} (the first step) serves for solving the linear systems \eqref{eq:lineqMGN} and \eqref{eq:lineqMGN2} in a stable and fast way (see Remark~\ref{rem:A_g_C_a}). Theorem \ref{th:gamma} shows that the order of the condition number of the matrix $\bfA_g$ is $\Theta(N^{t})$, where $t$ is the maximal multiplicity of roots of the characteristic polynomial  $g(\bfa)$ \eqref{eq:pol_z} on the unit circle. It is worth to mention that the use of Algorithm \ref{alga:fourier_basis_A_comp} increases the accuracy of computing the diagonal elements of $\bfA_g$ and does not change its condition number.

\paragraph{The VPGN method}

In the VPGN algorithm, the inverted matrix $\bm\Gamma(\bfa)$ is used in Algorithm~\ref{alg:solution_calc_vp} (for calculating the expression \eqref{eq:spZa_kostya}) and in Algorithm~\ref{alg:gauss_newton_vp} (for computing the expression \eqref{eq:vpformula}). For fast inversion, the diagonals of the matrix $\bm\Gamma(\bfa)$ are computed explicitly; then the Cholesky factorization is used.
It is shown in \cite[Section 6.2]{Usevich2014} that the condition number of $\bm\Gamma(\bfa)$ is $O(N^{2t})$.
Thus, this implementation of the inversion of $\bm\Gamma(\bfa)$ in VPGN is less stable than the inversion of $\bfA_g$ in MGN, since the condition number of $\bfA_g$ is $\Theta(N^{t})$.

Certainly, the inversion of $\bm\Gamma(\bfa)$ can be performed with better stability.
For example, one can use the QR factorization of the matrix $\bfW^{-1/2} \bfQ(\bfa)$ instead of the inversion of $\bm\Gamma(\bfa)$. However, the QR factorization does not exploit the banded structure of matrix $\bm\Gamma(\bfa)$, therefore, it is significantly slower than the Cholesky factorization if $\bfW^{-1}$ is banded.
We do not compare the MGN and VPGN algorithms by stability if $\bfW^{-1}$ is not banded, since then the computational cost of the VPGN algorithm is very large. Thus, in fact, the practical case considered in this comparison is the case of a diagonal weight matrix $\bfW$.


\subsection{Numerical comparison} \label{subsec:speed}
We present the numerical comparison starting from comparing the algorithm's stability. First, we construct a special example for demonstrating stability/accuracy. Then, the same example will be used for comparing the computational costs.

\subsubsection{Stability} \label{subsec:basisacc}
With the help of Lemma~\ref{lemma:locminnec}, we construct an example, where a local solution of \eqref{eq:wls} is known.
For constructing a solution of rank $r=3$, we use the well-known theory about the relation of linear recurrence relations,
characteristic polynomials, their roots and the explicit form of the series, see e.g. the book \cite[Sections 3.2.1, 3.2.2]{Golyandina2013} with a brief description of this relation in the context of time series structure.

Let $\tsY_N^\star = (b x_1^2, \ldots, b x_N^2)^\rmT$, where $x_i$, $i=1,\ldots,N$, form the equidistant grid in $[-1; 1]$ and the constant $b$ is such that $\|\tsY_N^\star\|=1$. The series $\tsY_N^\star$ satisfies the GLRR($\bfa^*$) for $\bfa^* = (1, -3, 3, -1)^\rmT$. Since the last component of $\bfa^*$ is equal to $-1$, we can say that the series satisfies the LRR($\bfa^*$).
Denote $\widehat \tsN_N = (c |x_1|, \ldots, c |x_N|)^\rmT$, where the constant $c$ is such that $\|\widehat \tsN_N\| = 1$.
Construct the observed series as $\tsX_N = \tsY_N^\star + \tsN_N$, where $\tsN_N = \widehat \tsN_N - \Proj_{\calZ((\bfa^*)^2), \Sigminus} \widehat \tsN_N$.
Thus, the pair $\tsX_0 = \tsY_N^\star$ and $\tsX = \tsX_N$ satisfies the conditions of Lemma \ref{lemma:locminnec}, which provides the necessary conditions for local minima. The sufficient condition (the positive definiteness of the Hessian matrix of the objective function  $\|\tsX -  \tsS(\Si0, \Ai0)\|^2_\bfW$ \cite[Theorem 2.3]{nocedal2006numerical}) was tested numerically for $N < 100$.
Details of the example implementation see in Section~\ref{sec:alg_details}.

The comparison is performed for the methods VPGN, S-VPGN, MGN and S-MGN for different $N$ from $20$ to $50000$; the compensated Horner scheme is used within the algorithms.
For simplicity, consider the non-weighted case, when $\Sigminus$ is the identity matrix.

\begin{figure}[!hbt]
	\centering
	(a)\includegraphics[scale=1.00]{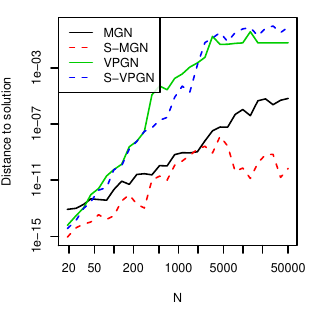}
	(b)\includegraphics[scale=1.00]{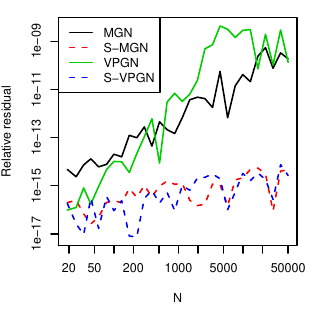}
	\caption{Comparison of algorithms by distance to the solution (a) and by relative residuals (b), for different $N$.}
	\label{fig:kostya_comp_disp}
\end{figure}

\begin{figure}[!hbt]
	\centering
	\includegraphics[scale=1.00]{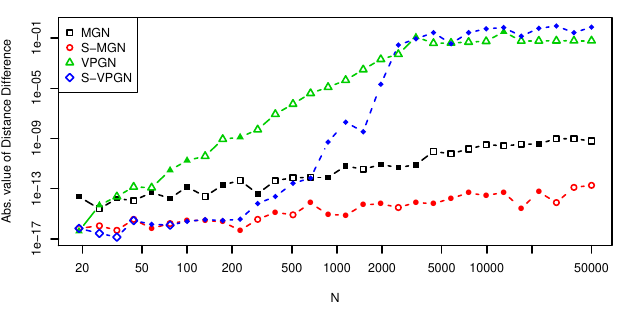}
	\caption{Comparison of algorithms by absolute differences between the values of the objective function at the final point of the algorithm and at the point of local minimum, for different $N$. The filled plotting symbols correspond to positive values of the differences.
}
	\label{fig:dist_diff}
\end{figure}

 Denote $\widetilde \tsY^\star$ the result of an algorithm participating in the comparison.
The main comparison was done by accuracy, that is, by the Euclidean distance between $\widetilde \tsY^\star$ and $\tsY^\star_N$ (Fig.~\ref{fig:kostya_comp_disp}(a)).
Also, we checked if the obtained solution $\widetilde \tsY^\star$ satisfies the GLRR($\bfa^\star$) used at the last iteration of the algorithm (Fig.~\ref{fig:kostya_comp_disp}(b)).
The measure of agreement with the GLRR($\bfa^\star$) is the relative residual ${\|\bfQ^\rmT(\bfa^\star)\widetilde \tsY^\star\|}\big/{\|\bfa^\star\|}$.
In addition, the algorithms were compared by discrepancy between the values of the objective function at the final
point of the algorithm and at the point of local minimum, i.e. by $\|\tsX_N -  \widetilde \tsY^\star\| - \| \tsX_N - \tsY^\star_N\|$ (Fig.~\ref{fig:dist_diff}).

The compared algorithms contain a line search in the descent direction $\Delta_k$.
 The line search method and the stopping criteria are not specified in the algorithms.
 The details of the used method are described in Section~\ref{sec:alg_details}. It is important that the method together with the stopping criterion is numerically stable
 with respect to the accuracy of computation.

The algorithms were started from the GLRR($\bfa_0$), where  $\bfa_0 =\bfa^* + 10^{-6} (1, 1, $ $ 1, 1)^\rmT$. Figure~\ref{fig:kostya_comp_disp}(a) shows that the accuracy of MGN and S-MGN is better than the accuracy of VPGN and S-VPGN. On the other hand, the resultant time series produced by the methods S-MGN and S-VPGN are close to $\overline \calD_r$ for all considered times-series lengths $N$ (see small relative residuals in Fig.~\ref{fig:kostya_comp_disp}(b)), whereas the methods MGN and VPGN yield time series
which are far from $\overline \calD_r$ for large $N$. Note that in exact arithmetic, VPGN and S-VPGN would produce the same results; the same is true for the pair of MGN and S-MGN.

Let us demonstrate the difference between VPGN and S-VPGN. Fig.~\ref{fig:dist_diff} shows that for most of $N$ the numerical solution provided by the VPGN method is closer to $\tsX_N$ than the theoretic solution is (the depicted differences are negative). This is an over-fitting, since the numerical solution $\widetilde \tsY^*$ is far from the series of rank $r$  for large $N$ (Fig.~\ref{fig:kostya_comp_disp}(b)). For S-VPGN, the difference is positive;
however, both VPGN and S-VPGN are further from the theoretical solution $\tsY^\star_N$ than MGN and S-MGN are.
It seems that negative values for S-MGN are explained not by an over-fitting but by the machine accuracy of numerical
calculations.

\subsubsection{Computational cost}
For effectively implemented algorithms, the computational speed should have the same order as the theoretical computational cost in FLOPs. Let us numerically confirm Remark~\ref{rem:compcost}.
We will consider the computational speed for different implementations of step 6 of Algorithms \ref{alg:gauss_newton_vp} and \ref{alg:gauss_newton_our}, where $\Delta_{k}$ is calculated. This speed characterizes the computational speed
of one iteration.
We consider different time series lengths $N$ and two types of the weight matrix $\bfW$, the identity matrix and a 3-diagonal matrix, which is the inverse of the autocovariance matrix of an autoregressive process of order 1. The speed is estimated with the help of the example described in Section~\ref{subsec:basisacc}.

The results for the CPU time are depicted in Fig.~\ref{fig:comp_time}. Since we compare asymptotic behavior (as $N\to\infty$), we eliminate the constant time, which does not depend on $N$, in the following way.
 For each algorithm, we consider the CPU times for different values of $N$ starting from 100 and then divide them by the CPU time for $N$ equal 100.
  Note that if $\bfW$ is diagonal, the computational times of the algorithms are asymptotically almost the same. However, if
  $\bfW$ contains three diagonals, the computational times for the methods MGN and S-MGN are much smaller than that for the methods VPGN and S-VPGN.

\begin{figure}[!hbt]
	\centering
	(a)\includegraphics[scale=1.00]{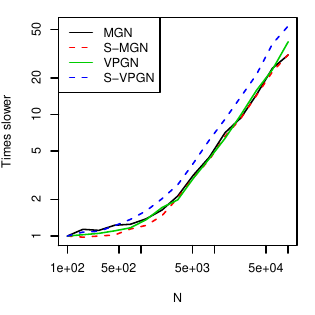}
	(b)\includegraphics[scale=1.00]{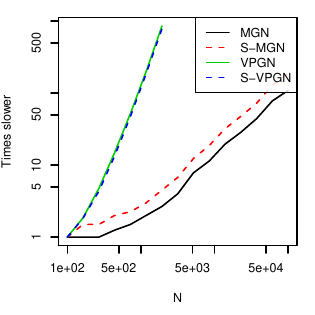}
	\caption{Comparison of algorithms by CPU times of one iteration for different $N$; (a) diagonal $\bfW$ and (b) 3-diagonal $\bfW$.}
	\label{fig:comp_time}
\end{figure}

\subsection{Signal estimation using MGN: with and without gaps} \label{subsec:model}
Consider a time series $\tsY_{50}$ similar to the one considered in \cite{Ishteva.etal2014}, which is the sum of a signal of rank $r=4$ and Gaussian white noise. That is, let the signal $\tsS_{50}$ have the following form: $\tsS_{50} = (s_1, \ldots, s_{50})^\rmT$, where
\begin{equation*}
s_i = 0.9^i \cos\left(\frac{\pi}{5} i\right) + \frac{1}{5} 1.05^i \cos \left(\frac{\pi}{12} i + \frac{\pi}{4}\right), \quad i = 1, \ldots, 50,
\end{equation*}
and
\begin{equation*}
\tsY_{50} = \tsS_{50} + 0.2 \frac{\tsN_{50}}{\|\tsN_{50}\|}\|\tsS_{50}\|;
\end{equation*}
here the series $\tsN_{50}$ consists of i.i.d. normal random variables with zero mean and unit standard deviation. Note that since $\tsY_{50}$ contains a random component, we were not able to reproduce the time series studied in \cite{Ishteva.etal2014} exactly.

Let us consider two versions of the time series $\tsY_{50}$, the first one is without missing data and the second time series with artificial gaps at positions $10 \ldots 19$ and $35 \ldots 39$, and construct two estimates of
the signal by the MGN method (Algorithm~\ref{alg:gauss_newton_our}).

In Algorithm~\ref{alg:gauss_newton_our}, the weight matrix $\bfW$ should be set. Since the noise is white,
the identity matrix $\bfW = \bfI_{50}$ was taken for the case without gaps; for the case with gaps, we changed ones on the diagonal of $\bfW$ at the positions of missing data to zeros.
For constructing the initial GLRR, we impute the mean value of the time series to replace the missing entries and then take
the GLRR coefficients from the last ($(r+1)$-th) right singular vector  of the SVD of the $(r+1)$-trajectory matrix
$\calT_{r+1}(\tsY_{50})$.

The results are presented in Figure \ref{fig:model}. The series $\tsY_{50}$ is indicated by the black dots, the signal $\tsS_{50}$ is depicted by the blue line, and the obtained approximation $\widetilde \tsS$ is shown by the red solid line. Note that in both cases $\widetilde \tsS$ gives a fairly close estimate of $\tsS_{50}$, despite even a big gap at $10 \ldots 19$ in the second case with missing values.

\begin{figure}[!hbt]
	\centering
	(a)\includegraphics[scale=1.00]{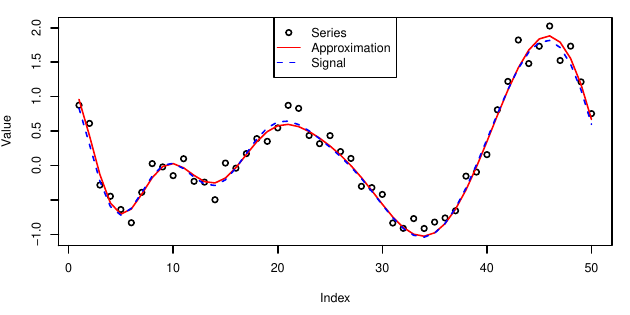}
	(b)\includegraphics[scale=1.00]{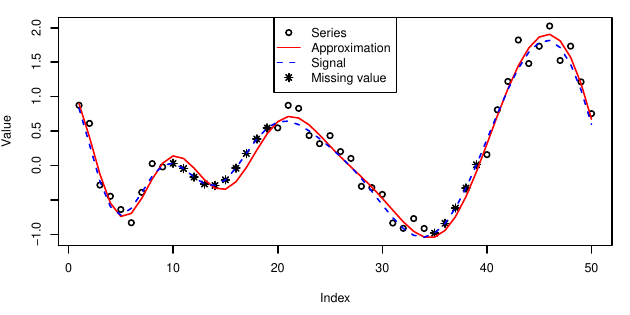}
	\caption{Estimates of the signal $\tsS_{50}$ using the MGN algorithm, (a) without gaps (b) with gaps.}
	\label{fig:model}
\end{figure}

\section{Conclusion}
\label{sec:conclusion}
In this paper we presented a new iterative algorithm (MGN, Algorithm~\ref{alg:gauss_newton_our}) for computing the numerical solution to the
problem~\eqref{eq:wls} and compared it with a state-of-art algorithm based on the variable projection approach (VPGN, Algorithm~\ref{alg:gauss_newton_vp}).
We showed that the proposed algorithm MGN allows the implementation, which is more numerically stable for the case of
multiple roots of the characteristic polynomial (in particular, for polynomial series, where the multiplicity
is equal to the polynomial degree plus one). This effect can be explained by the inversion of matrices with
condition number $O(N^{t})$ in MGN (Theorem~\ref{th:gamma}), where $t$ is the multiplicity, while the direct implementation
of VPGN deals with matrices with condition number $O(N^{2t})$ \cite[Section 6.2]{Usevich2014}.
The comparison of computational costs in Section~\ref{sec:comp_cost} shows that the algorithm MGN has slightly larger costs for the case of banded inverses $\bfW^{-1}$ of weight matrices. However, in the case of autoregressive noise with covariance matrix $\bm\Sigma$,
the corresponding weight matrix $\bfW = \bm\Sigma^{-1}$ is banded itself and $\bfW^{-1}$ is not banded.
Then the proposed algorithm MGN has a much lower computational cost in comparison with VPGN.
An important feature of the MGN algorithm is that it can be naturally extended to the case of missing data
without increasing the computational cost (see Remark~\ref{rem:wlsseminorm} and
the example in Section~\ref{subsec:model}.

To construct and justify the new algorithm, the properties of the space of low-rank time series were studied.
These properties can be useful not only in the framework of the algorithm justification. In particular,
we proved (Theorem~\ref{th:tangent}) that the tangent subspace at the point $\tsS$, which is governed by a GLRR($\bfa$), can be described in terms of the GLRR($\bfa^2$).
This fact allows one to construct first-order linear approximations to functions at points from $\calD_r$.
Then, in Section~\ref{sec:ZofA} we present a numerically stable algorithm of projecting a series to the set $\calZ(\bfa)$ of time series, which are governed by the GLRR($\bfa$). This can be useful for numerical solutions of different approximation problems related to the SLRA problems.


\section*{Acknowledgments} The reported study was funded by RFBR, project number 20-01-00067.

\clearpage
\appendix
\label{sec:app}

\section{Theoretical details}
\subsection{Correspondence between notations}
\label{sec:notation}
For the convenience of comparisons, in Table~\ref{table:defines} we present the correspondence between the notation used in this paper and the notation from  \cite{Usevich2012, Usevich2014}.

\begin{table}[!hh] 	\centering
	\caption{Correspondence between notations}
    \label{table:defines}
	\begin{tabular}{|c|c|c|c|c|c|c|c|c|}
		\hline
		This paper & $\tsX$ & $N$ & $r+1$ & $1$ & $N-r$ & $\bfa$ & $\Sigminus$ & $\bm\Gamma(\bfa)$ \\ \hline
		Usevich \& Markovsky & $p_D$ & $n_p$ & $m$ & $d$ & $n$ & $R$ & $W$ & $\Gamma$\\
		\hline
	\end{tabular}
\end{table}

\subsection{Rank of \eqref{eq:model}}
\label{sec:rank_calc}
\begin{proposition}
Let a series $\tsS$ of length $N$ have the form \eqref{eq:model}, $0\le \omega_k \le 0.5$, $0\le \phi_k < 2\pi$ and $m_k\ge 0$ be the polynomial degree, $k=1,\ldots,d$.
Suppose that the pairs $(\alpha_k, \omega_k)$ are different.
Also, assume that if $\omega_k = 0$ or $\omega_k = 0.5$, then $\phi_k \neq 0$.
Let $r_k$ be equal to 2 if $0 < \omega_k < 0.5$ and be equal to 1 otherwise.
Then the rank of $\tsS$ is equal to $\sum_{k=1}^d (m_k+1)r_k$ for sufficiently large $N$.
\end{proposition}

\begin{proof}
The assertion about the ranks of real-valued time series is the consequence of the analogous results for complex-valued time series.
If a series $\tsC=(c_1,\ldots,c_N)$ has terms $c_n = \sum_{k=1}^s P_{m_k}(n) \mu_k^n$ with different complex $\mu_k$, then its rank $r$ is
equal to  $\sum_{k=1}^s (m_k+1)$. This directly follows from the explicit form of the basis of the column space of the trajectory matrix $\calT_{r+1}(\tsC)$, which consists of  $r$ linearly independent vectors
$(1^i \mu_k^1, 2^i \mu_k^2, \ldots (r+1)^i \mu_k^{r+1})^\rmT$, $k=1,\ldots,s$, $i=0,\ldots,m_k$.
The rank of a real-valued time series is induced by the presentation of $\exp(\alpha_k n) \sin(2\pi \omega_k n + \phi_k)$, $0<\omega_k < 0.5$, as a linear combination
of $\mu^n$ and $\overline{\mu}^n$, where $\mu = \exp(\alpha_k + \unit 2\pi \omega_k)$ and $\overline{\mu}$ is the complex conjugate to $\mu$.
\end{proof}

\subsection{Proof of lemma about $\overline{\calD_r}$}
\label{sec:closure}
\begin{lemma}
	$\tsS \in \overline{\calD_r}$ if and only if $\tsS$ is governed by a GLRR($\bfa$) defined by a vector $\bfa \in \spR^{d+1}$, $d \le r$.
\end{lemma}
\begin{proof}
Let us consider the set of matrices $\calM_{\le r} \subset \spR^{L \times K}$ of rank not larger than $r$, and $\calM_{=r}$ the set of matrices of rank  $r$. Fix $L = r+1$, $K = N - L + 1 = N-r$.
Denote $\widehat \calD_r = \{\tsX\in \spR^N : \rank \calT_{r+1}(\tsX) \le r \} = \calT_{r+1}^{-1}(\calM_{\le r} \cap \calH) = \bigcup_{s=1}^r \calD_{s}$.
By definition,  $\calD_r = \{\tsX : \rank \calT_{r+1}(\tsX) = r \} = \calT_{r+1}^{-1}(\calM_{=r} \cap \calH)$.

It is known that $\overline{\calM_{=r}} = \calM_{\le r}$, see \cite{Lewis2008}. Thus, we have $\overline{\calD_r} = \overline{\calT_{r+1}^{-1}(\calM_{=r} \cap \calH)} = \calT_{r+1}^{-1}(\overline{\calM_{=r} \cap \calH}) \subset \calT_{r+1}^{-1}(\overline{\calM_{=r}} \cap \calH) = \calT_{r+1}^{-1}(\calM_{\le r} \cap \calH) = \widehat \calD_r$.

	To prove $\overline{\calD_r} = \widehat \calD_r$, we show that any $\tsS \in \widehat \calD_r$ can be approximated by a series $\tsX \in \calD_r$ with arbitrary precision. Let $\tilde{r} < r$ and $\tsS \in \calD_{\tilde{r}}$ satisfy a GLRR($\tilde{\bfa}$), $\tilde{\bfa} = (a_1, \ldots, a_{\tilde{r}+1})^\rmT \in \spR^{\tilde{r} + 1}$. It is sufficient to show that we can approximate $\tsS$ by $\tsX \in \calD_{\tilde{r} + 1}$; then we can obtain an approximating series from $\calD_r$ by subsequent approximations with ranks increased by 1.

Let us take such real $\mu$ that the series $\tsD=(\mu,\mu^2,\ldots,\mu^N)^\rmT$ of rank 1 is not governed by the GLRR($\tilde{\bfa}$). Denote $\bfd_M = (\mu,\mu^2,\ldots,\mu^M)^\rmT$; then $\tsD=\bfd_N$. For any real $\alpha \ne 0$, we have
	$\tsX(\alpha) = \tsS + \alpha \tsD \in \widehat \calD_{\tilde{r} + 1}$, since the series $\tsX(\alpha)$ is governed by the GLRR($\bfb$) with $\bfb = (\mu a_1, \mu a_2 - a_1, \mu a_3 - a_2, \ldots, \mu a_{\tilde{r}+1} - a_{\tilde{r}}, -a_{\tilde{r} + 1})^\rmT \in \spR^{\tilde{r} + 2}$. Thus, $\rank \tsX(\alpha) \le \tilde{r} +1$.
	
	Now let us show that $\rank \tsX(\alpha) \ge \tilde{r} +1$.
 We need to show that $\rank \calT_{\tilde{r} + 1}(\tsX(\alpha)) = \tilde{r}+1$ for any $\alpha \ne 0$.
 Due to \cite[Corollary 8.1]{Marsaglia1974}, it is enough to show that the column and row spaces of $\calT_{\tilde{r} +1}(\tsS)$ and $\calT_{\tilde{r} +1}(\alpha \tsD)$ have empty intersection.
 We know that $\colspace\left(\calT_{\tilde{r} +1}(\alpha \tsD)\right) = \sspan(\bfd_{\tilde{r} +1})$ and $\rowspace\left(\calT_{\tilde{r} +1}(\alpha \tsD)\right) = \sspan(\bfd_{N-\tilde{r}})$. Also, note that a vector $\bfv$ belongs to $\colspace\left({\calT_{\tilde{r} + 1}(\tsS)}\right)$ if and only if $\tilde{\bfa}^\rmT \bfv = 0$, and a vector $\bfu$ belongs to $\rowspace\left({\calT_{\tilde{r} + 1}(\tsS)}\right)$ if and only if
  $\left(\bfQ^{N-\tilde{r}, \tilde{r}}(\tilde{\bfa})\right)^\rmT \bfu= \bm{0}_{N-2\tilde{r}}$. However, by construction of $\tsD$, $\tilde{\bfa}^\rmT \bfd_{\tilde{r} +1} \ne 0$, and $\left(\bfQ^{N-\tilde{r}, \tilde{r}}(\tilde{\bfa})\right)^\rmT \bfd_{N-\tilde{r}}  \ne \bm{0}_{N-2\tilde{r}} $. Therefore, we have
 $\rowspace\left(\calT_{\tilde{r} + 1}(\tsS)\right) \cap \rowspace\left(\calT_{\tilde{r} + 1}(\alpha \tsD)\right) = \emptyset$ and $\colspace\left( \calT_{\tilde{r} + 1}(\tsS)\right) \cap \colspace\left(\calT_{\tilde{r} + 1}(\alpha \tsD)\right) = \emptyset$. The lemma is proved, since $\alpha$ can be an arbitrarily small positive number.
\end{proof}

    \section{Proofs of propositions from the paper}
    \subsection{Proof of Theorem~\ref{th:parametrization} and Proposition~\ref{prop:parametrization}}
    \label{sec:th:parametrization}

    \begin{proof}
    The first statement of Proposition~\ref{prop:parametrization} will provide the parameterizing mapping introduced in Theorem~\ref{th:parametrization}
    if we prove the correctness of \eqref{eq:param} and \eqref{eq:param_rev}, the uniqueness of $S_\tau$ satisfying relations of Theorem \ref{th:parametrization}, then prove that $S_\tau$ is an injective mapping and \eqref{eq:param_rev} defines the inverse of the mapping $S_\tau$ given in \eqref{eq:param}.

    Let us prove the correctness of \eqref{eq:param}.
    To begin with, we show that the matrix $\bfZ_{\row{\calI({\i0})}}$ is not singular and therefore invertible. This will be a consequence of non-singularity of $(\bfZ_0)_{\row{\calI({\i0})}}$ for any basis of $\calZ(\bfa_0)$.

    Let us represent $\bfa_0$ as $\bfa_0 = (0, \ldots, 0, b_{r_m +1 }, \ldots, b_{1}, 0, \ldots, 0)^\rmT\in \spR^{r+1}$, with $r_b$ zeroes at the beginning and $r_e$ zeroes at the end, $r_e + r_b + r_m = r$. Let us construct a matrix $\bfZ_0^\star= [\bfZ_\text{begin}: \bfZ_\text{middle}:\overline{} \bfZ_\text{end}]$ consisting of three blocks: $\bfZ_\text{begin} = \begin{pmatrix}
    \bfI_{r_b} \\
    \bm{0}_{(N - r_b)\times r_b}
    \end{pmatrix}$, $\bfZ_\text{middle} = \begin{pmatrix}
    \bm{0}_{r_b \times r_m} \\
    \widehat \bfZ_\text{middle} \\
    \bm{0}_{r_e \times r_m}
    \end{pmatrix} $, $\bfZ_\text{end} = \begin{pmatrix}
    \bm{0}_{(N - r_e)\times r_e} \\
    \bfI_{r_e}
    \end{pmatrix}$, where the columns of the matrix $\widehat \bfZ_\text{middle} \in \spR^{(N - r_b - r_e) \times r_m}$ form a basis of the space of time series of length $N - r_b - r_e$ governed by the LRR with coefficients $-b_2/b_1, \ldots, -b_{r_m +1 }/b_1$. Since $\{1, \ldots, r_b\} \cup \{N - {r_e}+1, \ldots, N\} \subset \calI(\i0)$ and any submatrix of size $r_m \times r_m$ of $\widehat \bfZ_\text{middle}$ is non-degenerate \cite[Prop. 2.3]{Usevich2010}, we obtain the non-degeneracy of $(\bfZ_0^\star)_{\row{\calI({\i0})}}$. Any other matrix which consists of basis vectors of $\calZ(\bfa_0)$ can be represented in the form $\bfZ_0^\star \bfP$ with a non-singular matrix $\bfP \in \spR^{r\times r}$. Therefore, matrix $(\bfZ_0^\star \bfP)_{\row{\calI({\i0})}}$ is also non-degenerate.

    Now let us prove the non-degeneracy of $\bfZ_{\row{\calI({\i0})}}$. Since $\calZ(\bfa)$ is the orthogonal complement to $\calQ(\bfa)$, $\Proj_{\calZ(\bfa)}$ can be represented as a continuous function $\Proj_{\calZ (\bfa)} = \bfI_N - \Proj_{\bfQ(\bfa)}$ of $\bfa$, $\bfa \ne \bm{0}_{r+1}$, where $\bfQ(\bfa)$ is defined in \eqref{op:Q}.
   Note that the determinant of $\bfZ_{\row{\calI({\i0})}}$ is a continuous function of $\bfZ$. In turn, $\bfZ$ continuously depends on $\Ai0$.
   Since $\bfZ(\bfa_0) = \bfZ_0$ and the determinant of $(\bfZ_0)_{\row{\calI({\i0})}}$ is non-zero, there is a neighborhood of $(\bfa_0)_{\calK(\i0)}$, such that
   the determinant of $\bfZ_{\row{\calI({\i0})}}$ is not zero; therefore, the matrix $\bfZ_{\row{\calI({\i0})}}$ is invertible.

    The constructed mapping \eqref{eq:param} does not depend on $\bfZ_0$. Indeed, for any non-singular matrix $\bfP\in \spR^{r\times r}$:
    $\left(\Proj_{\calZ (\bfa)} \bfZ_0 \bfP\right) \left((\Proj_{\calZ (\bfa)} \bfZ_0 \bfP)_{\row{\calI({\i0})}}\right)^{-1} = \bfZ \left(\bfZ_{\row{\calI({\i0})}}\right)^{-1}$.

	Let us demonstrate that the properties of $S_\tau$, which are stated in Theorem~\ref{th:parametrization}, are fulfilled; i.e., show that $\tsS \in \calD_r$, the series $\tsS$ satisfies the GLRR($\bfa$) and $(\tsS)_{\vecrow{\calI(\i0)}} = \Si0$. The series $\tsS$ satisfies the GLRR($\bfa$), since each column of the matrix $\bfZ$ satisfies the GLRR($\bfa$).
	To prove that $\tsS \in \calD_r$, consider the matrix $\calT_{r+1}(\tsS_0)$ and choose a submatrix of size $r \times r$ with
non-zero determinant. Then take the submatrix $\bfB$ of the matrix $\calT_{r+1}(S_\tau(\Si0, \Ai0))$ with the same location. Its determinant is
 a continuous function of $(\Si0, \Ai0)$, since the function given in \eqref{eq:param} is continuous. Therefore, there exists a neighborhood
 of $\left((\tsS_0)_{\calI(\i0)}, (\bfa_0)_{\calK(\i0)}\right)^\rmT$, where the determinant of $\bfB$ is non-zero; thus, $\tsS \in \calD_r$.
The condition $(\tsS)_{\vecrow{\calI(\i0)}} = \Si0$ is fulfilled, since
    \begin{equation*}
    (\tsS)_{\vecrow{\calI(\i0)}} =
    \left(\bfZ_{\row{\calI({\i0})}} \left(\bfZ_{\row{\calI({\i0})}}\right)^{-1}\right) \Si0 = \Si0.
    \end{equation*}

    Let us explain the uniqueness of the mapping $S_\tau$ satisfying the relations of Theorem~\ref{th:parametrization}.
    Let  $\widehat S$ be a different mapping satisfying the relations of Theorem~\ref{th:parametrization}, $\widehat \tsS = \widehat S(\Si0, \Ai0) \in \calD_r$. We know that $\widehat \tsS \in \calZ(\bfa)$. Therefore, columns of $\bfZ$ contain a basis of $\calZ(\bfa)$. Let $\widehat \tsS = \bfZ \bfv$ and $\bfv \in \spR^r$ be the coefficients of the expansion of $\widehat \tsS$ in the columns of $\bfZ$. Then the following is fulfilled: $(\bfZ \bfv)_{\calI({\i0})} = \Si0$. However, $\bfZ_{\row{\calI({\i0})}} \bfv = \Si0$ together with the invertibility of $\bfZ_{\row{\calI({\i0})}}$ leads to $\bfv =  \left(\bfZ_{\row{\calI({\i0})}}\right)^{-1} \Si0$. Therefore, $\widehat \tsS = \left(\bfZ \left(\bfZ_{\row{\calI({\i0})}}\right)^{-1}\right) \Si0 = \tsS$.

    Let us prove that $S_\tau$ is an injective mapping. We choose two different sets of parameters $\big(\Si0^{(1)}, \Ai0^{(1)}\big)^\rmT$, $\big(\Si0^{(2)}, \Ai0^{(2)}\big)^\rmT$ in the vicinity of $\left((\tsS_0)_{\calI(\i0)}, (\bfa_0)_{\calK(\i0)}\right)^\rmT$ and consider $\tsX_1 = S_\tau\big(\Si0^{(1)}, \Ai0^{(1)}\big)$, $\tsX_2 = S_\tau\big(\Si0^{(2)}, \Ai0^{(2)}\big)$. If $\Si0^{(1)} \ne \Si0^{(2)}$, then $\tsX_1 \ne \tsX_2$, since $(\tsX_1)_{\vecrow{\calI(\i0)}} \ne (\tsX_2)_{\vecrow{\calI(\i0)}}$. Let $\Si0^{(1)} = \Si0^{(2)}$ be fulfilled, but $\Ai0^{(1)} \ne \Ai0^{(2)}$. This means that the orthogonal complements $\sspan(\fullop(\Ai0^{(1)}))$ and $\sspan(\fullop(\Ai0^{(2)}))$ to $\colspace{\left(\calT_{r+1}(\tsX_1)\right)}$ and $\colspace{\left(\calT_{r+1}(\tsX_2)\right)}$ respectively are different and therefore these column spaces differs.
    Thus, $\tsX_1\ne \tsX_2$.

    Let us prove the correctness of \eqref{eq:param_rev}.
    According to the statement of Proposition~\ref{prop:parametrization}, $\Ai0$ defined in \eqref{eq:param_rev} is obtained from a renormalization of $\hat \bfa = \hat \bfa(\tsS)$ such that the $\i0$-th element becomes equal to $-1$.
    Let us prove the correctness of this definition of $\Ai0$, i.e., the possibility to renormalize $\hat \bfa$. Consider the matrix $\bfS = \calT_{r+1}(\tsS) \in \spR^{(r+1)\times (N-r)}$. Let $\calJ$ be a subset of indices such that the submatrix $(\bfS_0)_{\col{\calJ}}\in \spR^{(r+1)\times r}$ has rank $r$, where $\bfS_0 = \calT_{r+1}(\tsS_0)$. Then $\Proj_{\calL(\tsS)}$ can be represented as a continuous function $\Proj_{\calL(\tsS)} = \Proj_{\bfS_{\col{\calJ}}}$ in the vicinity of $\tsS_0$; therefore, we can choose a neighborhood of $\tsS_0$ in which $\hat a_\i0$ does not vanish.

    Let us explain that $\eqref{eq:param_rev}$ gives the inverse of the mapping $S_\tau$. Let $\tsS = S_\tau(\Si0, \Ai0)$. The values $\Si0 = (\tsS)_{\vecrow{\calI({\i0})}}$ are taken directly from the time series. The series $\tsS$ is governed by the GLRR($\hat \bfa$) since the vector $\hat \bfa$ is orthogonal to $\colspace(\calT_{r+1}(\tsS))$ by its definition. But the series $\tsS$ is governed by the GLRR($\bfa$); hence, $\bfa$ coincides with $\hat \bfa$ up to normalization. Therefore, renormalization of $\hat \bfa$ gives us the required $\Ai0$. This consideration concludes the proof.
    \end{proof}

    \subsection{Proof of Theorem \ref{th:param_smooth}}
    \label{sec:th:param_smooth}
    \begin{proof}
    	We need to show that $\Proj_{\calL(\tsS)}$ and $\Proj_{\calZ(\bfa)}$ from Proposition \ref{prop:parametrization} are smooth projections in the vicinity of $\bfS_0$ and $\bfZ_0$ respectively.
    	
    	Since $(\bfS_0)_{\col{\calJ}}$ has full rank, $\Proj_{\calL(\tsS)} = \bfS_{\col{\calJ}} \left(\left(\bfS_{\col{\calJ}}\right)^\rmT \bfS_{\col{\calJ}}\right)^{-1} \bfS^\rmT_{\col{\calJ}}$ is a smooth function in the vicinity of $\bfS_0$.
    	Since $\bfQ(\bfa)$ has full rank, see definition \eqref{op:Q}, $$\Proj_{\calZ(\bfa)} = \bfI_N - \bfQ(\bfa)\left(\bfQ^\rmT(\bfa) \bfQ(\bfa)\right)^{-1} \bfQ^\rmT(\bfa)$$ is smooth everywhere except $\bfa = \bfzero_{r+1}$.
    	
    	It is clearly seen that the other mappings involved in the parameterization are smooth in the corresponding vicinities.
    \end{proof}

\subsection{Proof of Theorem~\ref{th:tangent}}
\label{sec:th:tangent}

Let us start with two lemmas.
    It is convenient to separate the parameters ($2r$ arguments of the mapping $S_\tau$) into two parts, $\Si0$ and $\Ai0$.
    Then $\bfJ_{S_\tau} = [ \bfF_\bfs:\bfF_\bfa ]$, where $\bfF_\bfs = (\bfJ_{S_\tau})_{\col{\{1, \dots, r\}}}$, $\bfF_\bfa = (\bfJ_{S_\tau})_{\col{\{r+1, \dots, 2r\}}}$. Let $\bfa = \fullop(\Ai0)$.

    \begin{lemma}
    	\label{eqa:derivS}
    	$\bfQ^\rmT(\bfa) \bfF_\bfs = \bm{0}_{(N-r) \times r}$; $\colspace(\bfF_\bfs) = \calZ(\bfa)$.
    \end{lemma}	
    \begin{proof}
    	Let $\bfF_\bfs = [F_{s,1}:\ldots:F_{s,r}]$. Consider the equality $\bfQ^\rmT(\bfa) S_\tau(\Si0, \Ai0) = \bfzero_{N-r}$ and differentiate it with respect to $(\Si0)_{\vecrow{(i)}}$. We obtain $\bfQ^\rmT(\bfa) F_{s, i} = \bfzero_{N-r}$, which means that $\colspace(\bfF_\bfs) \subset \calZ(\bfa)$.
    	The fact $(\bfF_{\bfs})_{\row{\calI({\i0})}} = \bfI_r$ completes the proof.
    \end{proof}

    \begin{lemma}
    	\label{eqa:derivA}
    	$\bfQ^\rmT(\bfa) \bfF_{\bfa} = \bfM$, where $\bfM = - (\bfS_{\row{\calK({\i0})}})^\rmT$ and $\bfS = \calT_{r+1}(\tsS)$;
    	$\colspace(\bfF_\bfa) \subset \calZ(\bfa^2)$.
    \end{lemma}	
    \begin{proof}
    	Let $\bfF_\bfa = [F_{a,1}:\ldots:F_{a,r}]$. Consider the equality $\bfQ^\rmT(\bfa) S_\tau(\Si0, \Ai0) = \bfzero_{N-r}$ and differentiate it with respect to $(\Ai0)_{\vecrow{(i)}}$, i.e. $i$-th element of $\bfa_{\calK({\i0})} = \Ai0$, $i = 1, \ldots, r$. Then we obtain $\bfQ^\rmT(\bfe_{j}) \tsS + \bfQ^\rmT(\bfa) F_{a, i} = \bfzero_{N-r}$, where $\bfe_j\in \spR^{r+1}$ and $j=(\calK({\i0}))_i$ is $i$-th element of $\calK({\i0})$. (Note that $\bfQ^\rmT(\bfe_{j}) \tsS$ is the $j$-th column of the transposed $(r+1)$-trajectory matrix $\bfS^\rmT$.) Therefore, the equation $\bfQ^\rmT(\bfa) \bfF_{\bfa} = - (\bfS_{\row{\calK({\i0})}})^\rmT$ is proved.
    	
    	To prove the second statement of the lemma, let us take the matrix $\bfQ^{N-r, r}(\bfa) \in \spR^{N \times (N - r)}$. Due to the first statement, the equality $\left(\bfQ^{N-r, r}(\bfa)\right)^\rmT \left(\bfQ^{N, r}(\bfa)\right)^\rmT \bfF_{\bfa} = \bm{0}_{(N-2r) \times r}$ is valid. From \cite[Sections 2.1 and 2.2]{Usevich2017} it follows that then we have $\left(\bfQ^{N-r, r}(\bfa)\right)^\rmT \left(\bfQ^{N, r}(\bfa)\right)^\rmT = \bfQ^\rmT(\bfa^2)$. Therefore, $\bfQ^\rmT(\bfa^2) \bfF_{\bfa} = \bm{0}_{(N-2r) \times r}$.
    \end{proof}

    Now we can prove Theorem \ref{th:tangent}.
    \begin{proof}
    	It follows from Lemma \ref{eqa:derivS} that
    	\begin{equation*}
    	\bfQ^\rmT(\bfa^2) \bfF_{S} = (\bfQ^{N-r, r}(\bfa))^\rmT \bfQ^\rmT(\bfa) \bfF_{S} = \bm{0}_{(N-2r) \times r}.
    	\end{equation*}
    	Therefore,  $\colspace(\bfJ_{S_\tau}) \subset \calZ(\bfa^2)$. Also, $\tsS \in \calZ(\bfa^2)$. Since we have a diffeomorphism at the point $\tsS$, the Jacobian matrix $\bfJ_{S_\tau}$ has full rank $2r$. Hence, $\colspace(\bfJ_{S_\tau}) = \calZ(\bfa^2)$.
    \end{proof}

     \subsection{Proof of Lemma \ref{lemma:minindr}}
    \label{sec:lemma:minindr}
    \begin{proof}
    	Assume the contrary. Denote $\tsS^\star = \tsS_0$ a point of global minimum in the problem \eqref{eq:wls} and assume that $\tsS_0 \in \calD_{r_0}$, $r_0 < r$, is such that $\tsS_0$ satisfies a GLRR($\bfa_0$), $\bfa_0 = (a_1, \ldots, a_{r_0+1})^\rmT \in \spR^{r_0 + 1}$. Construct $N$ linearly independent exponential series $\tsS^{(i)}$ of length $N$, $\tsS^{(i)} = (e^{\lambda_i}, e^{2 \lambda_i}, \ldots, e^{N \lambda_i})^\rmT$, which are governed by the GLRR($\bfa^{(i)}$) with $\bfa^{(i)} = (e^{\lambda_i}, -1)$, $i=1,\ldots,N$, where all $\lambda_i$ are different. Then for any real $\alpha$ we have
    	$\tsS_0 + \alpha \tsS^{(i)} \in \overline{\calD_r}$ since the series $\tsS_0 + \alpha \tsS^{(i)}$ is governed by the GLRR($\bfb_i$) with $\bfb_i = (e^{\lambda_i} a_1, e^{\lambda_i} a_2 - a_1, e^{\lambda_i} a_3 - a_2, \ldots, e^{\lambda_i} a_{r_0+1} - a_{r_0}, -a_{r_0 + 1})^\rmT \in \spR^{r_0 + 2}$.
    	
    	Denote $\langle \tsZ, \tsY \rangle_{\Sigminus} = \tsZ^\rmT \Sigminus \tsY$ the weighted inner product  corresponding to the norm $\| \cdot \|_{\Sigminus}$. By the condition of the lemma, $\tsX - \tsS_0 \ne \bm{0}_N$. Consider the inner products $\langle \tsX - \tsS_0, \tsS^{(i)} \rangle_{\Sigminus}$, $i = 1, 2, \ldots, N$. Since $\tsS^{(i)}$, $i=1,\ldots,N$, form a basis of $\spR^N$, there exists an index $j$ such that $\langle \tsX - \tsS_0, \tsS^{(j)} \rangle_{\Sigminus} \ne 0$. Let us take $\tsS_1 = \tsS_0 + \frac{\langle \tsX - \tsS_0, \tsS^{(j)} \rangle_{\Sigminus}}{\langle \tsS^{(j)}, \tsS^{(j)} \rangle_{\Sigminus}} \tsS^{(j)}$ governed by the GLRR($\bfb_i$) (hence, $\tsS_1$ belongs to $\overline{\calD_r}$), and show that $\|\tsX - \tsS_1\|_{\Sigminus} < \|\tsX - \tsS_0\|_{\Sigminus}$. Indeed,
    	\begin{equation*}
    	\langle \tsX - \tsS_0, \tsX - \tsS_0  \rangle_{\Sigminus} - \langle \tsX - \tsS_1, \tsX - \tsS_1  \rangle_{\Sigminus} = \frac{\left( \langle \tsX - \tsS_0, \tsS^{(j)} \rangle_{\Sigminus} \right)^2}{\langle \tsS^{(j)}, \tsS^{(j)} \rangle_{\Sigminus}}>0.
    	\end{equation*}
    	We obtain a contradiction with the initial suggestion that $\tsS_0 = \tsS^\star$ is a point of global minimum in the problem \eqref{eq:wls}.
    \end{proof}

\subsection{Proof of Theorem~\ref{th:equivalency}}
\label{sec:th:equivalency}
Let us fix the iteration number $k$. 
Denote by $\bfF_\bfs = \left(\bfJ_{S_\tau} \right)_{\col{\{1, \ldots, r\}}}$ the first $r$ columns of the Jacobian matrix $\bfJ_{S_\tau} = \bfJ_{S}(\Si0^{(k)}, \Ai0^{(k)})$, and by $\bfF_\bfa = \left(\bfJ_{S_\tau} \right)_{\col{\{r+1, \ldots, 2r\}}}$ the last $r$ columns of $\bfJ_{S_\tau}$.

\begin{proof}
Let us rewrite the weighted pseudoinverse in the \eqref{eq:iterGNA} as
\begin{equation*}
	\left( \winverse{\bfJ_{S_{\tau}}(\Si0^{(k)}, \Ai0^{(k)})}{\bfW} \big(\tsX - S_{\tau}^\star(\Ai0^{(k)})\big) \right) _{\col{\{r+1, \ldots, 2r\}}}
	= \left( \inverse{\left(\bfW^{1/2} \bfJ_{S_{\tau}}(\Si0^{(k)}, \Ai0^{(k)})\right)} \bfW^{1/2}\big(\tsX - S_{\tau}^\star(\Ai0^{(k)})\big) \right) _{\col{\{r+1, \ldots, 2r\}}}.
\end{equation*}	
  Applying the Frisch-Waugh-Lovell theorem \cite{Lovell2008simple} about the partitioned regression to the obtained pseudoinverse for regressors $\bfW^{1/2} \bfF_{\bfs}$ and $\bfW^{1/2} \bfF_{\bfa}$, we get the following sequence of equalities:
\begin{multline*}
\left( \inverse{\left(\bfW^{1/2} \bfJ_{S_{\tau}}(\Si0^{(k)}, \Ai0^{(k)})\right)} \bfW^{1/2}\big(\tsX - S_{\tau}^\star(\Ai0^{(k)})\big) \right) _{\col{\{r+1, \ldots, 2r\}}} = \\
\inverse{\left( (\bfI_N - \Proj_{\bfW^{1/2} \bfF_{\bfs}})\bfW^{1/2} \bfF_{\bfa} \right)}(\bfI_N - \Proj_{\bfW^{1/2} \bfF_{\bfs}}) \bfW^{1/2} (\tsX - S_{\tau}^\star(\Ai0^{(k)})) = \\
\winverse{(\bfI_N - \Proj_{\calZ(\fullop(\Ai0^{(k)})), \bfW})\bfF_{\bfa}}{\bfW}(\bfI_N - \Proj_{\calZ(\fullop(\Ai0^{(k)})), \bfW}) (\tsX - \Proj_{\calZ(\fullop(\Ai0^{(k)})), \bfW}(\tsX)).
\end{multline*}	
Since $\bfI_N - \Proj_{\calZ(\fullop(\Ai0^{(k)})), \bfW}$ is a projector, $(\bfI_N - \Proj_{\calZ(\fullop(\Ai0^{(k)})), \bfW})^2 = \bfI_N - \Proj_{\fullop(\calZ(\bfa^{(k)})), \bfW}$. Thus, we obtain the following iteration equivalent to \eqref{eq:iterGNA}:
\begin{equation}
\Ai0^{(k+1)}
= \Ai0^{(k)} + \gamma \winverse{(\bfI_N - \Proj_{\calZ(\fullop(\Ai0^{(k)})), \bfW})\bfF_{\bfa}}{\bfW}(\bfI_N - \Proj_{\calZ(\fullop(\Ai0^{(k)})), \bfW})\tsX.
\end{equation}
By Lemma~\ref{eqa:derivA}, $\bfQ^\rmT(\fullop(\Ai0^{(k)})) \bfF_\bfa = \bfM$. By the theorem's conditions, $\bfQ^\rmT(\fullop(\Ai0^{(k)})) \widehat \bfF_\bfa = \bfM$. Thus, $\bfQ^\rmT(\fullop(\Ai0^{(k)}))(\bfF_\bfa - \widehat \bfF_\bfa) = \bfzero_{(N-r)\times r}$. Since $\calQ(\bfa)$ is the orthogonal complement to $\calZ(\bfa)$, $(\bfI_N - \Proj_{\calZ(\fullop(\Ai0^{(k)})), \bfW})(\bfF_\bfa - \widehat \bfF_\bfa) = \bfzero_{N\times r}$, which finishes the proof.
\end{proof}


\subsection{Proof of Theorem~\ref{th:gamma}}
\label{sec:th:gamma}
\begin{proof}

Denote by $\upangle{x}{y}$ the angle between two points on the complex unit circle $\spT$, $0 \le \upangle{x}{y} \le \pi$. Let us prove the first statement. Consider a root $z_1 \in \spT$ of multiplicity $t_1$, $t_1\le t$, of the polynomial $g_\bfa(z)$; then for any $\alpha$ we have $\min_{w \in \calW(\alpha)} \upangle{w}{z_1} \le \frac{\pi}{N}$ by the Dirichlet principle. Let us fix any $0 \le \alpha_0 < 2 \pi$ and choose $w_0 = \argmin_{w \in \calW(\alpha_0)} \upangle{w}{z_1}$. Since $|z_1 - w_0|=O(1/N)$, we have $|\lambda_\text{min}(\alpha)| \le |g_\bfa(w_0)| = O(N^{-t})$.
	
	To prove the second statement, let us find any point $x \in \spT$ for which $|g_\bfa(x)| = \max_{z \in \spT} |g_\bfa(z)| > 0$ is fulfilled. Again, by the Dirichlet principle, we have $\min_{w \in \calW(\alpha)} \upangle{w}{x} \le \frac{\pi}{N}$ for any $\alpha$. Let us choose $w_1 = \argmin_{w \in \calW(\alpha_0)} \upangle{w}{x}$. Since $|x - w_1|=O(1/N)$ and $g_\bfa(z)$ is continuous, we have $|\lambda_\text{max}(\alpha)| \ge |g_\bfa(w_1)| = \Omega(1)$, which with $|\lambda_\text{max}(\alpha)| = O(1)$ proves the second part.
	
	To prove the third statement, let us construct a piecewise approximation of $g_\bfa(z)$ in $z$. Consider the decomposition $g_\bfa(z) = p_\bfa(z) q_\bfa(z)$, where the roots of $p_\bfa(z)$ belong to $\spT$ while the roots of $q_\bfa(z)$ do not. By construction, $\inf_{z \in \spT} |q_\bfa(z)| > 0$.
	
	Let $z_1, \ldots, z_k$ be the roots of $p_\bfa(z)$ with multiplicities $t_1, \ldots, t_k$. We split the circle $\spT$ into $k$ semi-open non-intersecting arcs $\calS_1, \ldots, \calS_k$, $\spT = \bigcup_{1 \le i \le k} \calS_i$, such that $z_i \in \calS_i$ for any $i$ and $z_j \notin \overline{\calS_i}$ for any $j \ne i$ ($\overline{\calS_i}$ denotes the closure of $\calS_i$), which leads to $\inf_{z \in \calS_i} \left| {p_\bfa(z)}/{(z-z_i)^{t_i}} \right| > 0 $.
	
	To finish the proof we need to show that there exists $0 \le \alpha = \alpha(N) < 2 \pi$
	such that we have $\min\limits_{w \in \calW(\alpha(N)), \; 1 \le i \le k} \upangle{w}{z_i} = \Theta(1/N)$.
	Let us denote for $0 \le \mu < \pi / N$ and $z \in \spT$
	$$\calB_{z, \mu} = \{0 \le \alpha < 2 \pi : \min_{w \in \calW(\alpha)} \upangle{w}{z} \le \mu \}.$$
	The set $\calB_{z, \mu}$ has the explicit form:
	$$\calB_{z, \mu} = \bigcup_{0 \le j \le N-1} \Big\{ \Arg \left( \exp \left( \unit \left( \frac{2 \pi j}{N} + y \right)  / z \right) \Big|_{-\mu \le y \le \mu} \right) \Big\}.$$

	Let us comment this expression. Consider $\omega_j^{(\alpha)} = \exp \left(\unit \left(\frac{2 \pi j}{N} - \alpha \right) \right)$ and choose $\alpha_j$ such that $\upangle{\omega_j^{(\alpha_j)}}{z} \le \mu$. This means that the polar angle of the ratio $\omega_j^{(\alpha_j)}/z$ belongs to the interval $[-\mu, \mu]$, i.e. $\omega_j^{(\alpha_j)}/z \in \{ \exp \left( \unit x \right) |_{-\mu \le x \le \mu} \}$. Let us perform equivalent transformations:
$$\exp \left(\unit \left({2 \pi j}/{N} - \alpha_j \right) \right) \in \{ z \exp \left( \unit x \right) |_{-\mu \le x \le \mu} \},$$
 $$\exp \left(\unit \left(\alpha_j - {2 \pi j}/{N} \right) \right) \in \{ \exp \left( \unit y \right) / z |_{-\mu \le y \le \mu} \},$$
  where $y = -x$. Finally, note that $\alpha_j \in  \calM_j = \left\{ \Arg (\exp \left( \unit ({2 \pi j}/{N} + y) \right) / z ) |_{-\mu \le y \le \mu} \right\}$. The inequality $\min_{w \in \calW(\alpha)} \upangle{w}{z} \le \mu$ is valid if $\alpha$ is equal to one of $\alpha_0, \ldots, \alpha_j$.
Therefore, the union of all such sets $\calM_j$ for $j = 0, \ldots, N-1$ gives us $\calB_{z, \mu}$.
	
	The Lebesgue measure of $\calB_{z, \mu}$ is equal to $\mes \calB_{z, \mu} = 2 \mu N$ for $\mu < \pi/N$.
	Let us take $\mu = \frac{\pi}{2Nk}$ and consider $\calB = \bigcup_{1 \le i \le k} \calB_{\mu, z_i}$.
	Since $\mes \calB \le \pi$, we obtain $\mes \widehat \calB \ge \pi$ for $\widehat \calB = [0; 2 \pi) \setminus \calB$,
	which means that $\widehat \calB$ is not the empty set.
	Thus, we have proved that for any $\alpha \in \widehat \calB$
	$$\min_{w \in \calW(\alpha), \; 1 \le i \le k} \upangle{w}{z_i} > \frac{\pi}{2Nk}.$$
	
	Let us fix an arbitrary $\alpha_0 \in \widehat \calB$ and consider any $w \in \calW(\alpha_0)$. 
	For each $i$ such that $w \in \calS_i$, $|w - z_i| = \Theta(1/N)$.
	Then $|g_\bfa(w)| = |q_\bfa(w)| \left| \displaystyle{\frac{p_\bfa(w)}{(w-z_i)^{t_i}}} \right| |(w-z_i)^{t_i}| \ge C \Theta(N^{-t_i})$,
	where $C > 0$ is some constant.
\end{proof}

\section{Details of algorithms}
\subsection{Formulas for calculation of the iteration step \eqref{eq:gauss_simple} in VPGN} \label{sec:MUdetails}
An explicit form of the step \eqref{eq:gauss_simple} is contained in \cite[Proposition 3]{Usevich2014}.
Here we write down the formulas in our notation and also present a new form for the Jacobian $\bfJ_{\tsS_\tau^\star}$, which is more convenient for implementation.
\begin{lemma}
	\label{th:varproj}
	Let $\bfW$ be positive definite. The projection $\Proj_{\calZ(\bfa), \bfW}$ can be calculated as
	\begin{equation} \label{eq:spZaa_kostya}
	\Proj_{\calZ(\bfa), \bfW} \tsX = \left( \bfI_N - \bfW^{-1} \bfQ(\bfa) \bm\Gamma^{-1}(\bfa) \bfQ^\rmT(\bfa) \right) \tsX,
	\end{equation}
	where $\bm\Gamma(\bfa) = \bfQ^\rmT(\bfa) \Sigminus^{-1} \bfQ(\bfa)$.\\
	The columns of $\bfJ_{\tsS_\tau^\star}$ has the form
	\begin{equation} 	\label{eq:vpformulaa}
	(\bfJ_{\tsS_\tau^\star})_{\col{i}} = -\bfW^{-1} \bfQ(\bfa) \bm\Gamma^{-1}(\bfa) \bfQ^\rmT(\bfe_j) \Proj_{\calZ(\bfa), \bfW} \tsX
    -\Proj_{\calZ(\bfa), \bfW} \bfW^{-1} \bfQ(\bfe_j) \bm\Gamma^{-1}(\bfa) \bfQ^\rmT(\bfa) \tsX,
	\end{equation}
	where $\bfa = \fullop(\Ai0)$ and $j = (\calK({\i0}))_i$ is $i$-th element of $\calK({\i0})$.
	
\end{lemma}
\begin{proof}
    The equality
	\begin{equation*}
	S_\tau^\star(\Ai0) = \Proj_{\calZ(\fullop(\Ai0)), \bfW}(\tsX)
	\end{equation*}
	corresponds to the solution of the following quadratic problem:
	\begin{equation}
    \label{eq:lin_constraint}
	S_\tau^\star(\Ai0) = \argmin_{\substack{\tsY: \  \bfQ^\rmT(\bfa) \tsY=0}} \left(\frac{1}{2} \tsY^\rmT \bfW\tsY - \tsY \bfW \tsX \right).
	\end{equation}
    The problem \eqref{eq:lin_constraint} is the equality-constrained quadratic optimization problem, which can be written as a linear system \cite[Section 16.1]{nocedal2006numerical}. The Schur-complement method described in \cite[Section 16.2]{nocedal2006numerical} provides the expression \eqref{eq:spZaa_kostya} after substituting the corresponding notation.
	
	Proof of equality \eqref{eq:vpformulaa} is done by taking derivatives of \eqref{eq:spZaa_kostya} with respect to $a_j$:
	\begin{multline*}
	(\Proj_{\calZ(\bfa), \bfW} \tsX)'_{a_j} = - \bfW^{-1} \bfQ(\bfe_j) \left(\bfQ^\rmT(\bfa) \bfW^{-1} \bfQ(\bfa)\right)^{-1} \bfQ^\rmT(\bfa) \tsX \\
	-\bfW^{-1} \bfQ(\bfa) \left(\bfQ^\rmT(\bfa) \bfW^{-1} \bfQ(\bfa)\right)^{-1} \bfQ^\rmT(\bfe_j) \tsX \\+  \bfW^{-1} \bfQ(\bfa) \left(\bfQ^\rmT(\bfa) \bfW^{-1} \bfQ(\bfa)\right)^{-1}  \\
\times \left(\bfQ^\rmT(\bfe_j) \bfW^{-1} \bfQ(\bfa) + \bfQ^\rmT(\bfa) \bfW^{-1} \bfQ(\bfe_j) \right) \left(\bfQ^\rmT(\bfa) \bfW^{-1} \bfQ(\bfa)\right)^{-1} \\ \times \bfQ^\rmT(\bfa) \tsX
	=-\bfW^{-1} \bfQ(\bfa) \left(\bfQ^\rmT(\bfa) \bfW^{-1} \bfQ(\bfa)\right)^{-1} \bfQ^\rmT(\bfe_j) \Proj_{\calZ(\bfa), \bfW} \tsX \\
    -\Proj_{\calZ(\bfa), \bfW} \bfW^{-1} \bfQ(\bfe_j) \left(\bfQ^\rmT(\bfa) \bfW^{-1} \bfQ(\bfa)\right)^{-1} \bfQ^\rmT(\bfa) \tsX.
	\end{multline*}
\end{proof}

\subsection{The compensated Horner scheme for calculation of polynomials in Algorithms~\ref{alg:fourier_basis_A} and \ref{alg:fourier_grad}}
\label{subsec:hornerscheme}
The Horner scheme is an algorithm for evaluating univariate polynomials in floating point arithmetic. The accuracy of the compensated Horner scheme \cite[Algorithm \code{CompHorner}]{Graillat2008} is similar to the one given by the Horner scheme computed in twice the working precision. Let us describe how the Horner scheme (we will consider its compensated version) can be used for calculating the basis of $\bfZ(\bfa)$ and the matrix $\widehat \bfF_{\bfa}$ with improved accuracy.

The Horner scheme can be directly applied in Algorithms \ref{alg:fourier_basis_A} and \ref{alg:fourier_grad} for calculating the polynomial $g_\bfa$.
Moreover, the Horner scheme can improve the accuracy of the calculation of $\bfU_r$ at step 4 of Algorithm~\ref{alg:fourier_basis_A}; this improvement is important if $\bfL_r$
is ill-conditioned.

To use the advantage of the Horner scheme, let us consider a new way of calculating the matrix $\bfU_r$.
Let $\bfO_r$ be such that $\bfL_r \bfO_r$ consists of orthonormal columns; $\bfO_r$ can be found by either the QR factorization or the SVD.
Then $\bfU_r = \bfA_g^{-1} (\bfR_r \bfO_r)$, where the matrix $\bfR_r$ is calculated at step 3 of Algorithm~\ref{alg:fourier_basis_A}.
Since $(\bfR_r)_{\row{k}} = \left(\exp\big(\frac{\unit 2 \pi r k}{N}\big), \exp\big(\frac{\unit 2 \pi (r-1) k}{N}\big), \ldots, \exp\big(\frac{\unit 2 \pi k}{N}\big)\right)$,
we can reduce the multiplication of $\bfR_r$ by a vector to the calculation of a polynomial of degree $r$ at the point $\exp\big(\frac{\unit 2 \pi k}{N}\big)$.
Therefore, we can accurately calculate the multiplication of $\bfR_r$ by a vector with the help of the Horner scheme. In particular, $\bfR_r \bfO_r$ can be calculated in this way.

\begin{algorithm}
	\caption{Calculation of the basis of $\calZ(\bfa)$ using the Compensated Horner Scheme}
	\label{alga:fourier_basis_A_comp}
    \Input{$\bfa \in \spR^r$.}
	\begin{algorithmic}[1]
		\State Compute $\alpha_0$ and $\bfA_g$ in the same way as at steps 1 and 2 of Algorithm~\ref{alg:fourier_basis_A} except for the use
of the algorithm \code{CompHorner} for calculation of values of the polynomials $g_\bfa$.
		\State Compute $\bfL_r$ and $\bfR_r$ in the same way as at step 3 of Algorithm~\ref{alg:fourier_basis_A}.
		\State Compute $\bfU_r$ in a new manner: find $\bfO_r$ such that $\bfL_r \bfO_r$ consists of orthonormal columns;
calculate $\bfB=\bfR_r \bfO_r$ by means of the algorithm \code{CompHorner}; calculate $\bfU_r = \bfA_g^{-1} (\bfB)$ directly by matrix multiplication.
\State\Return Matrix $\bfZ$, which is calculated in the same way as at steps 5 and 6 of Algorithm~\ref{alg:fourier_basis_A}, $\alpha_0$, $\bfA_g$.
	\end{algorithmic}
\end{algorithm}

Algorithm~\ref{alga:fourier_basis_A_comp} is a stable analogue of Algorithm~\ref{alg:fourier_basis_A}.
The stable version of Algorithm~\ref{alg:fourier_grad} differs by the change of the first step
``Compute $\alpha_0$, $\bfA_g$ using Algorithm \ref{alg:fourier_basis_A}'' to
``Compute $\alpha_0$, $\bfA_g$ using Algorithm \ref{alga:fourier_basis_A_comp}''.

\subsection{Computational details of the numerical example from Section~\ref{subsec:basisacc}}
\label{sec:alg_details}

\paragraph{Construction of the example} In Section~\ref{subsec:basisacc}, the example of an appropriate time series is theoretically constructed.
In practice, we should generate the time series $\tsX_N=(x_1,\ldots,x_N)^\rmT$ from this example with high numerical precision which is enough for comparing
the algorithms, which solve the problem \eqref{eq:wls}, by their accuracy.
The main difficulty lies in calculating the projection $\Proj_{\calZ(\bfa_0^2), \Sigminus}$.
The GLRR($\bfa_0$) with $\bfa_0 = (1, -3, 3, -1)^\rmT$ corresponds to the characteristic polynomial $g_{\bfa_0}(t)=(t-1)^3$ with the coefficients taken from $\bfa_0$.
Therefore, the GLRR($\bfa_0^2$) corresponds to the characteristic polynomial
$g^2_{\bfa_0}(t)=(t-1)^6$ and a basis of $\calZ(\bfa_0^2)$ consists of polynomials of degree not greater than $5$.
To obtain
the projection, we use Legendre polynomials \cite{Belousov2014} of degree from $0$ to $5$, which are calculated at the points $x_i$ as a basis of $\calZ(\bfa_0^2)$.
Then the constructed basis is orthogonalized.

\paragraph{Line search and stopping criteria} Let us provide details concerning the implementation of the line search at step 8 and the stopping criterion in Algorithms \ref{alg:gauss_newton_vp} and  \ref{alg:gauss_newton_our}.
We implemented the backtracking line search method \cite[Section 3.1]{nocedal2006numerical} in the direction $\Delta_k$ starting from the step size $\gamma = 1$ (the full step) and then dividing $\gamma$ by 2. The backtracking stops when
\begin{equation}
\label{eq:searchstop}
\|\tsX - \tsS^\star(\Ai0^{(k)} + \gamma \Delta_k) \|_{\bfW} \le \|\tsX - \tsS^\star(\Ai0^{(k)}) \|_{\bfW};
\end{equation}
 then $\gamma_k = \gamma$. If there is no such $\gamma$ for $\gamma = 1, 1/2, 1/4, \ldots, 2^{-16}$, then we set $\gamma_k = 0$.
  The stopping criterion of the whole algorithm is the equality $\gamma_k = 0$, which means that the current iteration can not improve the approximation to the solution. However, this approach has the following issue. Let for $\gamma = 1$, which corresponds to the full Gauss-Newton step, the relative change be very small, e.g.,
\begin{equation}\label{eq:searchthr}
\left\Vert \frac{\tsS^\star(\Ai0^{(k)} + \Delta_k) - \tsS^\star(\Ai0^{(k)})}{\tsS^\star(\Ai0^{(k)})} \right\Vert < \zeta,
\end{equation}
where $\zeta$ has the order of a square root of machine epsilon ($\zeta = 5\cdot 10^{-8}$ in the numerical experiments).  Then the backtracking line search with the stopping rule \eqref{eq:searchstop}
is unstable due to a poor accuracy of calculating the objective function $\|\tsX - \tsS^\star(\Ai0^{(k)} + \gamma \Delta_k) \|^2_{\bfW}$, which is caused by the calculation of ill-conditioned inner products.

Let us modify the line search in the direction $\Delta_k$ for the case when the condition \eqref{eq:searchthr} is valid.
Both MGN and VPGN methods can be considered in two ways, as iterations of the parameters $\Ai0^{(k)}$ and as iterations of the series $\tsS^\star(\Ai0^{(k)})$.
When \eqref{eq:searchthr} is fulfilled at the $k$-th iteration step, we do not realize the backtracking line search; instead, we make a choice between two step sizes: $\gamma_k=1$ (the full step) or $\gamma_k=0$ (which stops the whole algorithm), where the choice is performed with the help of the vectors of parameters.
Denote the difference between the vectors of parameters at adjacent iterations as  $\widetilde \Delta_k = \Ai0^{(k+1)} - \Ai0^{(k)}$, which coincides with the direction vector $\Delta_k$ when $\gamma_k = 1$ according to step 8 of the algorithms.
 Let \eqref{eq:searchthr} be fulfilled.
 If $k=0$, we perform the full step with $\gamma_k=1$.
 Otherwise, we compare $\| \widetilde  \Delta_{k} \|$ and $\| \widetilde  \Delta_{k-1} \|$. If $\| \widetilde  \Delta_{k} \| < \| \widetilde  \Delta_{k-1} \|$,
 then we set $\gamma_k = 1$; otherwise we put $\gamma_k = 0$ and stop the algorithm.
 Thus, we propose a combination of the line search at step 7 and the algorithm stopping criterion with improved accuracy and stability.

\end{document}